\documentclass[12pt, reqno]{amsart}
\usepackage{amsmath,amsfonts,amsbsy,amsgen,amscd,mathrsfs,amssymb,amsthm}
\usepackage[usenames,dvipsnames]{xcolor}
\usepackage[colorlinks=true,citecolor=red,linkcolor=blue]{hyperref}
\usepackage[a4paper,margin=1in]{geometry}
\usepackage{setspace}
\usepackage{amsmath}
\usepackage{bm}
\usepackage{color}
\usepackage{array}
\usepackage{makecell}
\allowdisplaybreaks[4]
\numberwithin{equation}{section}
\newtheorem{theorem}{Theorem}[section]
\newtheorem{definition}[theorem]{Definition}
\newtheorem{lemma}[theorem]{Lemma}
\newtheorem{corollary}[theorem]{Corollary}
\newtheorem{remark}[theorem]{Remark}
\newtheorem{proposition}[theorem]{Proposition}

\newtheorem{problem}[theorem]{Problem}

\title[The $p$-th dual Minkowski problem for $k$-torsional rigidity]{The $p$-th dual Minkowski problem for the $k$-torsional rigidity corresponding to a $k$-Hessian equation}

\author{Xia Zhao and Peibiao Zhao}


\begin{document}
\begin{abstract}
The study of the dual curvature measures [Y. Huang, E. Lutwak, D. Yang \& G. Y. Zhang, Acta. Math. 216 (2016): 325-388], which connects the cone-volume measure and Aleksandrov's integral curvature, and  has created a precedent for the theoretical research of the dual Brunn-Minkowski theory.

Motivated by the foregoing groundbreaking works, the present paper introduces the $p$-th dual $k$-torsional rigidity associated with a $k$-Hessian equation and establishes its Hadamard variational formula with $1\leq k\leq n-1$, which induces the $p$-th dual $k$-torsional measure. Further, based on the $p$-th dual $k$-torsional measure, this article, for the first time, proposes the $p$-th dual Minkowski problem of the $k$-torsional rigidity which can be equivalently converted to a nonlinear partial differential equation in smooth case:
\begin{align}\label{eq01}
f(x)=\tau(|\nabla h|^2+h^2)^{\frac{p-n}{2}}h_{\Omega}(x)|Du(\nu^{-1}_\Omega(x))|^{k+1}\sigma_{n-k}(h_{ij}(x)+h_\Omega(x)\delta_{ij}),
\end{align}
where $\tau>0$ is a constant, $f$ is a positive smooth function defined on $S^{n-1}$ and $\sigma_{n-k}$ is the $(n-k)$-th elementary symmetric function of the principal curvature radii. We confirm the existence of smooth non-even solution to the $p$-th dual Minkowski problem of the $k$-torsional rigidity for $p<n-2$ by the method of a curvature flow which converges smoothly to the solution of equation (\ref{eq01}). Specially, a novel approach for the uniform lower bound estimation in the $C^0$ estimation for the solution to the curvature flow is presented with the help of invariant functional $\Phi(\Omega_t)$. 
\end{abstract}
\maketitle
\vskip 20pt
\section{Introduction and main results}

The characterizing area measure $S_k(\Omega,\cdot)$ problem is referred to as the Christoffel-Minkowski problem: For a given integer $1\leq k\leq n-1$ and a finite Borel measure $\mu$ on an unit sphere $S^{n-1}$, what are the necessary and sufficient conditions such that $\mu$ is equal to the area measure $S_{k}(\Omega,\cdot)$ of a convex body. When $k=1$, it is the Christoffel problem which was once independently solved by Firey \cite{FI1} and Berg \cite{BE}. The case of $k=n-1$, the Christoffel-Minkowski problem is just the classical Minkowski problem: Given a non-zero finite Borel measure $\mu$ on $S^{n-1}$, under what the necessary and sufficient conditions on $\mu$, does there exist an unique convex body $\Omega$ such that the given measure $\mu$ is equal to the surface area measure $S(\Omega,\cdot)$? For $1<k<n-1$, it is a difficult and long-term open problem. Some important progress of the Christoffel-Minkowski problem was obtained by Guan and Guan \cite{GB} and Guan and Ma \cite{GM}, as well as \cite{GP, HC, ZR} and the other relevant references.

The $L_p$ form of the Minkowski problem is called the $L_p$ Minkowski problem which is posed by Lutwak \cite{LE0} with $p>1$. The $L_p$ Minkowski problem contains some special versions, when $p=1$, it is the classical Minkowski problem, the famous log-Minkowski problem \cite{BO} with $p=0$, and $p=-n$, it is the centro-affine Minkowski problem \cite{ZG}. Moreover, the solution of the $L_p$ Minkowski problem plays a key role in establishing the $L_p$ affine Sobolev inequality \cite{HC2, LE01}. Haberl, Lutwak, Yang and Zhang \cite{HC1} proposed and studied the even Orlicz Minkowski problem in 2010 which is a more generalized form because the Orlicz Minkowski problem is the classical Minkowski problem with $\varphi(s)=s$ and the $L_p$ Minkowski problem with $\varphi(s)=s^{1-p}$.

Recently, a very important pioneering work was born, Huang, Lutwak, Yang and Zhang \cite{HY} introduced the $q$-th dual curvature measure $\widetilde{C}_q(\Omega,\cdot)$ and the dual Minkowski problem, where the dual Minkowski problem can be stated below: Given a nonzero finite Borel measure $\mu$ on $S^{n-1}$, what are the necessary and sufficient conditions for the existence of a convex body $\Omega$ in $\mathbb{R}^n$ such that $\mu=\widetilde{C}_q(\Omega,\cdot)$. Two special cases of the dual Minkowski problem include the log-Minkowski problem for $q=n$ and the Aleksandrov problem when $q=0$. This work of \cite{HY} is a major development for the dual Brunn-Minkowski theory which would prompt scholars to study the dual Minkowski problem of various measures.

With the continuous development and enrichment of the Minkowski problems and their dual analogues, the Minkowski problem has inspired many other problems of a similar nature. Examples include the capacity Minkowski type problems which relates to the solution of boundary values problems \cite{JE, XJ, CO}, the Gaussian Minkowski problem \cite{HY1,LJQ}, the chord Minkowski problem \cite{LE02,ZX0}. In this article, we focus on the relevant Minkowski problem of the torsional rigidity which is related to the solution of boundary value problems. Among them, the torsional rigidity is essentially equivalent to the existence of a solution to the Laplace equation, while the $q$-torsional rigidity is essentially equivalent to the existence of a solution to the $q$-Laplace equation. In addition,
the value of this functional quantitatively describes the comprehensive ability of an object's cross-section to resist torsional deformation and store torsional strain energy when the internal stress reaches a mechanical equilibrium state under the action of torque. It profoundly reveals how the geometric shape of an object ultimately determines its macroscopic mechanical properties through a classical partial differential equation.
For convenience, we here only state the definition of the $q$-torsional rigidity. Let $\mathcal{K}^n$ be the collection of convex bodies in Euclidean space $\mathbb{R}^n$. The set of convex bodies containing the origin in their interiors in $\mathbb{R}^n$, we write $\mathcal{K}^n_o$. Moreover,  let $C^2_{+}$ be the class of convex bodies of $C^2$ with a positive Gauss curvature at the boundary. Let $\Omega\in\mathcal{K}^n$, the $q$-torsional rigidity $T_q(\Omega)$ \cite{CA1} with $q>1$ is defined by
\begin{align*}
\frac{1}{T_q(\Omega)}=\inf\bigg\{\frac{\int_{\Omega}|\overline{\nabla} U|^q dy}{[\int_{\Omega}|U|dy]^q}:U\in W_0^{1,q}(\Omega),\int_{\Omega}|U|dy>0\bigg\}.
\end{align*}
It is illustrated in \cite{BE0,HH} that the above functional has an unique minimizer $u\in W_0^{1,q}(\Omega)$ satisfying the following boundary value problem
\begin{align*}
\left\{
\begin{array}{lc}
\Delta_qu=-1\ \ \  \text{in} \ \ \ \Omega,\\
u=0,\ \ \ \ \ \ \ \ \text{on} \ \ \ \ \partial\Omega,\\
\end{array}
\right.
\end{align*}
where $$\Delta_qu\hat{=}{\rm div}(|\overline{\nabla} u|^{q-2}\nabla u)$$
is the $q$-Laplace operator.

Applying the integral by part to the $q$-Laplace equation, with the aid of Poho$\check{\textrm{z}}$aev-type identities \cite{PU}, the integral formula of $q$-torsional rigidity can be given by
\begin{align*}
T_q(\Omega)^{\frac{1}{q-1}}=&\frac{q-1}{q+n(q-1)}\int_{S^{n-1}}h(\Omega,x)d\mu^{tor}_{q}(\Omega,x)\\
\nonumber=&\frac{q-1}{q+n(q-1)}\int_{S^{n-1}}h(\Omega,x)|\overline{\nabla} u|^qdS(\Omega,x).
\end{align*}

When $q=2$, $T_q(\Omega)$ is the so-called torsional rigidity $T(\Omega)$ of a convex body $\Omega$ whose Minkowski problem was firstly studied by Colesanti and Fimiani \cite{CA1}. The Minkowski problem for the torsional rigidity was extended to the $L_p$ version by Chen and Dai \cite{CZM} who proved the existence of solutions for any fixed $p>1$ and $p\neq n+2$, Hu and Liu \cite{HJ01} for $0<p<1$. Li and Zhu \cite{LN} first developed and proven the Orlicz Minkowski problem $w.r.t.$ the torsional rigidity by the variational method and Hu, Liu and Ma \cite{HJ} obtained the smooth solution for this problem by a Gauss curvature flow. Huang, Song and Xu \cite{HY0} established the $L_p$ variational formula for the $q$-torsional rigidity with $q>1$. Hu and Zhang \cite{HJ2} established the functional Orlicz-Brunn-Minkowski inequality for the $q$-torsional rigidity. Following the work of Hu and Zhang in \cite{HJ2}, Zhao et al in \cite{ZX} have had a systematic investigation on this topic and proposed the Orlicz Minkowski problem for the $q$-torsional rigidity with $q>1$ and obtained its smooth non-even solutions by method of a Gauss curvature flow. Moreover, the authors further  in \cite{ZX1} have also posed and studied the $p$-th dual Minkowski problem for the $q$-torsional rigidity with $q>1$ and obtained the existence of smooth even solutions for $p<n(p\neq 0)$ and smooth non-even solutions for $p<0$ by the method of a Gauss curvature flow.

In the present paper, we will extend the dual Minkowski problem of the $q$-torsional rigidity (associated with a $q$-Laplace equation) to the dual Minkowski problem of the $k$-torsional rigidity which is related equivalently to solutions of a $k$-Hessian equation instead of the $q$-Laplace equation. It is believed that this research will contribute to the enrichment and development for the $k$-torsional rigidity in the dual Brunn-Minkowski theorey. Now, we recall and state firstly the concept of the $k$-torsional rigidity and its related contents as follows. We consider a  $k$-Hessian equation below:
\begin{align}\label{eq101}
\left\{
\begin{array}{lc}
S_k(D^2u)=1\ \ \  \text{in} \ \ \ \Omega,\\
u=0,\ \ \ \ \ \ \ \ \text{on} \ \ \ \ \partial\Omega,\\
\end{array}
\right.
\end{align}
where $\Omega$ is a bounded convex domain of $\mathbb{R}^n$ and $S_k(D^2u)$ is the $k$-elementary symmetric function of the eigenvalues of $D^2u$, $k\in\{1,\cdots,n\}$.

Notice that, when $k=1$ in (\ref{eq101}), it is the Laplace equation, while $k=n$ in (\ref{eq101}), it is the well-known Monge-Amp\`{e}re equation. For $k\geq 2$, the $S_k$ operator is fully nonlinear and it is not elliptic unless when it is restricted to a suitable class of admissible functions, the so-called $k$-convex functions (see Section \ref{sec2} for more details).

Next, we introduce the functional $T_k$ related to the equation (\ref{eq101}) which can be defined as follows (see \cite{Sa}):
\begin{align}\label{eq102}
\frac{1}{T_k(\Omega)}=\inf\bigg\{\frac{-\int_{\Omega} wS_k(D^2w) dy}{[\int_{\Omega}|w|dy]^{k+1}}:w\in \Phi _{k}^0(\Omega)\bigg\},
\end{align}
where $\Phi _{k}^0(\Omega)$ is the set of admissible functions that vanish on the boundary.

Note that $S_1(D^2u)=\Delta u$ and the functional $T(\Omega)$ related to $\Delta u$ is called the torsional rigidity of $\Omega$ which is defined by Colesanti \cite{CA0}, for this reason, $T_k(\Omega)$ is called the $k$-torsional rigidity of $\Omega$.

Consider the functional
\begin{align*}
J(w)=\frac{1}{k+1}\int_{\Omega}(-w)S_k(D^2w) dy-\int_{\Omega}w dy.
\end{align*}
From the works of Wang \cite{WXJ,WXJ1}, we know that $J$ has a minimizer $u\in \Phi _{k}^{0}(\Omega)$ which solves (\ref{eq101}) and also minimizers the quotient in (\ref{eq102}). Then from (\ref{eq101}) and Poho$\check{\textrm{z}}$aev identity \cite[Proposition 3 in Appendix A]{BR}, the $k$-torsional rigidity can be directly calculated as
\begin{align*}
T_k(\Omega)=\bigg(\frac{1}{k(n+2)}\int_{S^{n-1}}h(\Omega,x)|Du(\nu_\Omega^{-1}(x))|^{k+1}dS_{n-k}(\Omega,x)\bigg)^k.
\end{align*}
Denote $\widetilde{T}_k(\Omega)=(T_k(\Omega))^{\frac{1}{k}}$, namely,
\begin{align*}
\nonumber\widetilde{T}_k(\Omega)=\frac{1}{k(n+2)}\int_{S^{n-1}}h(\Omega,x)|Du(\nu_\Omega^{-1}(x))|^{k+1}dS_{n-k}(\Omega,x),
\end{align*}
where $u$ is the solution of (\ref{eq101}) on $\Omega$, $h(\Omega,\cdot)$ is the support function of $\Omega$, $\nu_\Omega$ is the Gauss map of $\partial\Omega$ (then $\nu_\Omega^{-1}(x)$ is the point on $\partial\Omega$ where the outer normal direction is $x$) and $S_{n-k}(\Omega,\cdot)$ denotes the $(n-k)$-th area measure of $\partial\Omega$. In particularly, when $k=1$, $S_{n-1}(\Omega,\cdot)=S(\Omega,\cdot)$ is just the classical surface area measure and $T(\Omega)$ is the torsional rigidity of $\Omega$. From the theory of convex bodies and differential geometry (see for example \cite{SC} and \cite{UR}), we see in this case that
\begin{align}\label{eq103}
dS_{n-k}(\Omega,x)=\sigma_{n-k}(h_{ij}+h\delta_{ij})dx,\quad x\in S^{n-1},
\end{align}
where $dx$ is the Lebesgue measure on $S^{n-1}$, $h_{ij}$ is the second covariant derivative of $h$ with respect to the local orthonormal frame $\{e_1,e_2,\cdots,e_{n-1}\}$ on $S^{n-1}$ and $\sigma_{n-k}(h_{ij}+h\delta_{ij})$ is the $(n-k)$-th elementary symmetric function of the eigenvalues of $(h_{ij}+h\delta_{ij})$ and $\delta_{ij}$ is the Kronecker delta. Thus
\begin{align}\label{eq104}
\widetilde{T}_k(\Omega)=\frac{1}{k(n+2)}\int_{S^{n-1}}h(\Omega,x)|Du(\nu_\Omega^{-1}(x))|^{k+1}\sigma_{n-k}(h_{ij}+h\delta_{ij})dx.
\end{align}
We notice that $T_k:\mathbb{R}^n\rightarrow \mathbb{R}_{+}$ is a positively homogeneous operator of degree $(n+2)k$.

Motivated by the works of the dual curvature measure and the dual Minkowski problem in \cite{HY} and the work  of the Minkowski problem to the $k$-torsional rigidity \cite{ZX2}, we focus on in the present paper considering the $p$-th dual Minkowski problem for the $k$-torsional rigidity with $1\leq k\leq n-1$ in the dual Brunn-Minkowski theory. Firstly, we give the definition of the $p$-th dual $k$-torsional measure.

\begin{definition}\label{def11} Let $\Omega\in\mathcal{K}^n_o$, $1\leq k\leq n-1$ and $p\in\mathbb{R}$. We define the $p$-th dual $k$-torsional measure in the following table:
\begin{table}[h!]
\centering
\caption{The case of different $p$ to the $p$-th dual $k$-torsional measure}
\begin{tabular}{ |m{1.5cm}|m{12cm}| c |}
      \hline
      \thead{$p\neq n$} & \thead{$\widetilde{Q}_{k,n-p}(\Omega,\eta)=\frac{1}{n-p}\int_{\alpha_\Omega^*(\eta)}\rho^{p+1-k}_{\Omega}(v)|Du(r_\Omega(v))|^{k+1}dv$} \\
      \hline
      \thead{$p=n$} &  \makecell{$\widetilde{Q}_{k,0}(\Omega,\eta)=\lim_{p\rightarrow n}\widetilde{Q}_{k,n-p}(\Omega,\eta)$ \ \ \ \ \ \ \ \ \ \ \ \ \ \ \ \ \ \ \ \ \ \ \ \ \ \ \ \ \ \ \ \\ $\ \ \ \ \ \ \ \ \ \ \ \ =\int_{\alpha_\Omega^*(\eta)}\log\rho_{\Omega}(v)\rho^{n+1-k}_{\Omega}(v)|Du(r_\Omega(v))|^{k+1}dv$}\\
      \hline
    \end{tabular}
\end{table}

for each Borel $\eta\subset S^{n-1}$ and $r(\Omega,v)=\rho(\Omega,v)v$, $\rho(\Omega,\cdot)$ is the radial function of $\Omega$, $\alpha_\Omega^*$ is the reverse radial Gauss image on $S^{n-1}$ and $dv$ is the spherical measure on $S^{n-1}$ (see Definition \ref{def34} for details).
\end{definition}

Naturally, the $p$-th dual $k$-torsional rigidity $\widetilde{Q}_{k,{n-p}}(\Omega)$ of $\Omega\in\mathcal{K}_o^n$ with $p\in\mathbb{R}$ and $1\leq k\leq n-1$ is denoted by
\begin{table}[h!]
\centering
\caption{The case of different $p$ to the $p$-th dual $k$-torsional rigidity}
\begin{tabular}{ |m{1.5cm}|m{12cm}| c |}
      \hline
      \thead{$p\neq n$} & \thead{$\widetilde{Q}_{k,n-p}(\Omega)=\frac{1}{n-p}\int_{S^{n-1}}\rho^{p+1-k}_{\Omega}(v)|Du(r_\Omega(v))|^{k+1}dv$} \\
      \hline
      \thead{$p=n$} &  \makecell{$\widetilde{Q}_{k,0}(\Omega)=\lim_{p\rightarrow n}\widetilde{Q}_{k,n-p}(\Omega)$ \ \ \ \ \ \ \ \ \ \ \ \ \ \ \ \ \ \ \ \ \ \ \ \ \ \ \ \ \ \ \ \\ $\ \ \ \ \ \ \ \ \ \ \ \ =\int_{S^{n-1}}\log\rho_{\Omega}(v)\rho^{n+1-k}_{\Omega}(v)|Du(r_\Omega(v))|^{k+1}dv$}\\
      \hline
    \end{tabular}
\end{table}

Then the  Minkowski problem of prescribing the $p$-th dual $k$-torsional measure can be described as:
\begin{problem}\label{pro12}
Let $1\leq k\leq n-1$ and $p\neq n$. Given a non-zero finite Borel measure $\mu$ on $S^{n-1}$, what are the necessary and sufficient conditions on $\mu$ such that there exists a convex body $\Omega\in\mathcal{K}_o^n$ whose the $p$-th dual $k$-torsional measure $\widetilde{Q}_{k,{n-p}}(\Omega,\cdot)$ is equal to the given measure $\mu$?
\end{problem}

In addition, we call the measure $\widetilde{Q}_{k,0}(\Omega,\cdot)$ is the dual log $k$-torsional measure, then the Minkowski problem of prescribing the dual log $k$-torsional measure is called the dual log Minkowski problem to the $k$-torsional rigidity which is stated as follows:

\begin{problem}\label{pro12+}
Let $1\leq k\leq n-1$. Given a non-zero finite Borel measure $\mu$ on $S^{n-1}$, what are the necessary and sufficient conditions on $\mu$ such that there exists a convex body $\Omega\in\mathcal{K}_o^n$ whose dual log $k$-torsional measure $\widetilde{Q}_{k,0}(\Omega,\cdot)$ is equal to the given measure $\mu$?
\end{problem}

\begin{remark}
We only discuss Problem \ref{pro12} with $p\neq n$ in this paper, and in subsequent article, we will discuss Problem \ref{pro12+} of $p=n$.
\end{remark}

If the given measure $\mu$ in Problem \ref{pro12} is absolutely continuous with respect to the Lebesgue measure and $\mu$ has a smooth density function $f:S^{n-1}\rightarrow (0,\infty)$, then according to (\ref{eq103}) and the Corfton formula
\begin{align*}
\int_{S^{n-1}}\rho^{n+1-k}(\Omega,v)dv=\int_{S^{n-1}}h(\Omega,x)dS_{n-k}(\Omega,x),
\end{align*}
solving Problem \ref{pro12} can be equivalently viewed as solving the following nonlinear partial differential equation on $S^{n-1}$:
\begin{align*}
f(x)=\frac{1}{n-p}\rho_\Omega^{p-n}h_{\Omega}|Du(\nu^{-1}_\Omega(x))|^{k+1}\sigma_{n-k}(h_{ij}(x)+h_\Omega(x)\delta_{ij}),
\end{align*}
equivalently,
\begin{align}\label{eq105}
f(x)=\frac{1}{n-p}(|\nabla h|^2+h^2)^{\frac{p-n}{2}}h_{\Omega}(x)|Du(\nu^{-1}_\Omega(x))|^{k+1}\sigma_{n-k}(h_{ij}(x)+h_\Omega(x)\delta_{ij}).
\end{align}
Here $h$ is the unknown function on $S^{n-1}$ to be found, $\nabla h$ and $h_{ij}$ denote the gradient vector and the Hessian matrix of $h$ with respect to an orthonormal frame on $S^{n-1}$.

If the factor
\begin{align*}
\frac{1}{n-p}(|\nabla h|^2+h^2)^{\frac{p-n}{2}}h_{\Omega}(x)
\end{align*}
is omitted in equation (\ref{eq105}), then (\ref{eq105}) will become the partial differential equation of
the Minkowski problem for $k$-torsional rigidity \cite{ZX2}. If only the factor $\frac{1}{n-p}(|\nabla h|^2+h^2)^{\frac{p-n}{2}}$ is omitted, then equation (\ref{eq105}) can be viewed as the partial differential equation of the logarithmic Minkowski problem to $k$-torsional rigidity. Moreover, when $p=n$, (\ref{eq105}) is the equation of the dual log Minkowski problem to the $k$-torsional rigidity.

In the present paper, we will investigate the smooth solutions to the normalized $p$-th dual Minkowski problem for the $k$-torsional rigidity with $p\neq n$ by the method of a curvature flow. Roughly speaking,  the Gauss curvature flow and the mean curvature flow are the two most common curvature flow methods, and are used to demonstrate the Minkowski problem and geometric inequalities, respectively. The Gauss curvature flow was first introduced and studied by Firey \cite{FI2} to model the shape change of worn stones. Since then, the Gauss curvature flow has been widely used to find
the smooth solutions of the various Minkowski problems, see \cite{CCQ,CH,LR,LYN}. In addition, the most crucial and difficult part in the study of mean curvature flows is the analysis of singularities. According to Huisken's classical theory \cite{HUI}, the mean convex surface will develop a first type singularity. For such singularities, their microstructure can be studied through the expansion process, and the final limit model is the self similar contraction solution. In this regard, the series of works by Colding and Minicozzi \cite{CO1,CO2} provides us with a profound and complete understanding of singularity structures. They established a profound connection between the singularity theory of mean curvature flow and the theory of stable minimal surfaces, and provided a detailed characterization of singularity classification.

The normalized equation we will study in this paper is as follows:
\begin{align}\label{eq106}
f(x)=\tau(|\nabla h|^2+h^2)^{\frac{p-n}{2}}h_{\Omega}(x)|Du(\nu^{-1}_\Omega(x))|^{k+1}\sigma_{n-k}(h_{ij}(x)+h_\Omega(x)\delta_{ij}),
\end{align}
where $\tau$ is a positive constant.

Let $1\leq k\leq n-1$, $p\neq n$, $\partial\Omega_0$ be a smooth, closed and strictly convex hypersurface in $\mathbb{R}^n$ containing the origin in its interior and $f$ be a positive smooth function on $S^{n-1}$. We construct and consider the long-time existence and convergence of a following curvature flow which is a family of convex hypersurfaces $\partial\Omega_t$  parameterized by  smooth maps $X(\cdot ,t):
S^{n-1}\times (0, \infty)\rightarrow \mathbb{R}^n$
satisfying the initial value problem:
\begin{align}\label{eq107}
\left\{
\begin{array}{lc}
\frac{\partial X(x,t)}{\partial t}=\frac{\langle X,v \rangle^2}{f(x)}(|\nabla h|^2+h^2)^{\frac{p-n}{2}}|D u(X(x,t),t)|^{k+1}\sigma_{n-k}(x,t)v
-\eta(t)X(x,t),\\
X(x,0)=X_0(x),\\
\end{array}
\right.
\end{align}
where $\sigma_{n-k}(x,t)$ is the $(n-k)$-th ($1\leq k\leq n-1$) elementary symmetric function for principal curvature radii, $v$ is the
outer unit normal at $X(x,t)$, $\langle X,v \rangle$ represents the standard inner product of $X$ and $v$ in $\mathbb{R}^n$ and $\eta(t)$ is given by
\begin{align}\label{eq108}
\eta(t)=\frac{\int_{S^{n-1}}\rho(v,t)^{p+1-k}|D u(X,t)|^{k+1}dv}{\int_{S^{n-1}}f(x)dx}.
\end{align}

For convenience, we construct a following functional which is very important for $C^0$ estimate of the solution to curvature flow (\ref{eq107}).
\begin{align}\label{eq109}
\Phi(\Omega_t)=\int_{S^{n-1}}\log h(x,t)f(x)dx.
\end{align}
Note that, we will show that $\log h(x,t)$ is well-defined in Section \ref{sec5}, i.e. $h(x,t)>0$.

We obtain the long-time existence and convergence of the flow (\ref{eq107}) in this article, see Theorem \ref{thm13} for details.
\begin{theorem}\label{thm13}
Let $1\leq k\leq n-1$, $p<n-2$ and $u(\cdot,t)$ be a smooth admissible solution of (\ref{eq101}) in $\Omega_t$. Let $\partial\Omega_0$ be a smooth, closed and strictly convex hypersurface in $\mathbb{R}^n$ containing the origin in its interior, and $f$ be a positive smooth function on $S^{n-1}$. Then the flow (\ref{eq107}) has an unique smooth non-even convex solution $\partial\Omega_t=X(S^{n-1},t)$. Moreover, when $t\rightarrow\infty$, there is a subsequence of $\partial\Omega_t$ that converges in $C^{\infty}$ to a smooth, closed and strictly convex hypersurface $\partial\Omega_\infty$, the support function $h^\infty(x)$ of convex body $\Omega_\infty$ enclosed by $\partial\Omega_\infty$ satisfies equation (\ref{eq106}).
\end{theorem}

This paper is organized as follows. We collect some necessary background materials about convex bodies in Section \ref{sec2}. In Section \ref{sec3}, we obtain some properties of the $p$-th dual $k$-torsional measure and establish a Hadamard variational formula for the $p$-th dual $k$-torsional rigidity. In Section \ref{sec4}, we give the quantitative equation of the flow (\ref{eq107}) and confirm two key features of two important geometric functionals along the flow (\ref{eq107}). In Section \ref{sec5}, we give the priori estimates for solution to the flow (\ref{eq107}). We obtain the long-time existence and convergence of the flow (\ref{eq107}) and complete the proof of Theorem \ref{thm13} in Section \ref{sec6}.

\section{\bf Preliminaries}\label{sec2}
In this subsection, we give a brief review of some relevant notions and terminologies.

\subsection{Convex hypersurface} (see \cite{SC,UR}) Let $\mathbb{R}^n$ be the $n$-dimensional Euclidean space and $S^{n-1}$ be the unit sphere in $\mathbb{R}^n$. The origin-centered unit ball $\{y\in\mathbb{R}^n:|y|\leq 1\}$ is always denoted by $\mathcal{B}$. We write $\omega_n$ for the volume of $\mathcal{B}$ and denote its surface area by $n\omega_n$.

Let $\partial\Omega$ be a smooth, closed and strictly convex hypersurface in $\mathbb{R}^n$ containing the origin in its interior. The support function of a convex body $\Omega$ enclosed by $\partial\Omega$ is defined by
\begin{align*}h_\Omega(x)=h(\Omega,x)=\max\{x\cdot y:y\in\Omega\},\quad \forall x\in S^{n-1},\end{align*}
and the radial function of $\Omega$ with respect to $o$ (origin) $\in\mathbb{R}$ is defined by
\begin{align*}\rho_{\Omega}(v)=\rho(\Omega,v)=\max\{c>0:cv\in\Omega\},\quad  v\in S^{n-1}.\end{align*}
We easily obtain that the support function is homogeneous of degree $1$ and the radial function is homogeneous of degree $-1$.

For a convex body $\Omega\in\mathbb{R}^n$, its support hyperplane with outward unit normal vector $x\in S^{n-1}$ is represented by
\begin{align*}
H(\Omega,x)=\{y\in\mathbb{R}^n:y\cdot x=h(\Omega,x)\}.
\end{align*}
A boundary point of $\Omega$ which only has one supporting hyperplane is called a regular point, otherwise, it is a singular point. The set of singular points is denoted as $\sigma \Omega$, it is
well known that $\sigma \Omega$ has spherical Lebesgue measure 0.
The Gauss map $\nu_\Omega:y\in\partial \Omega\setminus \sigma \Omega\rightarrow S^{n-1}$ is represented by
\begin{align*}\nu_\Omega(y)=\{x\in S^{n-1}:y\cdot x=h_\Omega(x)\}.\end{align*}
Here $\partial \Omega\setminus \sigma \Omega$ is abbreviated as $\partial^\prime\Omega$.

Correspondingly, for a Borel set $\eta\subset S^{n-1}$, its inverse Gauss map is denoted by $\nu_\Omega^{-1}$,
\begin{align*}\nu_\Omega^{-1}(\eta)=\{y\in\partial^\prime \Omega:\nu_\Omega(y)\in\eta\}.\end{align*}

Suppose that $\Omega$ is parameterized by the inverse Gauss map $X:S^{n-1}\rightarrow \Omega$, that is $X(x)=\nu_\Omega^{-1}(x)$. Then the support function $h$ of $\Omega$ can be computed by
\begin{align}\label{eq201}h(x)=x\cdot X(x) , \ \ x\in S^{n-1},\end{align}
where $x$ is the outer unit normal of $\Omega$ at $X(x)$. Let $\{e_1,\cdots,e_{n-1}\}$ be an orthogonal frame on $S^{n-1}$. Let $\nabla$ be the gradient on $S^{n-1}$. Differentiating (\ref{eq201}), we have
\begin{align*}
\nabla_ih=\langle \nabla_ix,X(x)\rangle + \langle x,\nabla_iX(x)\rangle.
\end{align*}
Since $\nabla_iX(x)$ is tangent to $\partial\Omega$ at $X(x)$, we have
\begin{align*}
\nabla_ih=\langle \nabla_ix,X(x)\rangle.
\end{align*}
It follows that
\begin{align}\label{eq202}
\overline{\nabla} h=\nabla h+hx=X(x).
\end{align}
$\overline{\nabla}h$ is the point on $\partial\Omega$ whose outer unit normal vector is $x\in S^{n-1}$.

Denote by $h_i$ and $h_{ij}$ the first and second order covariant derivatives of $h$
on $S^{n-1}$, then computing as in \cite{JE0}, one can get
\begin{align}\label{eq203}
X(x)=h(x)_ie_i+h(x)x,\ \ \ \ X_i(x)=\omega_{ij}e_j,
\end{align}
where $\omega_{ij}=h_{ij}+h\delta_{ij}$. Note that we use the summation convention for the repeated indices here and after.

\subsection{Wull shapes and convex hulls}
Denote by $C(S^{n-1})$ the set of continuous functions on $S^{n-1}$ which is often equipped with the metric induced by the maximal norm. We write $C^+(S^{n-1})$ for the set of strictly positive functions in $C(S^{n-1})$. For any nonnegative $f\in C(S^{n-1})$, the Aleksandrov body is defined by
\begin{align*}
[f]=\bigcap_{v\in S^{n-1}}\bigg\{y\in \mathbb{R}^n:y\cdot v\leq f(v)\bigg\},
\end{align*}
the set is Wulff shape associated with $f$. Obviously, $[f]$ is a compact convex set containing the origin. If $\Omega$ is a compact convex set containing the origin, then $\Omega=[h_{\Omega}]$. The Aleksandrov convergence lemma is shown as follows: if the sequence $f_i\in C^+(S^{n-1})$ converges uniformly to $f\in C^+(S^{n-1})$, then $\lim_{i\rightarrow\infty}[f_i]=[f]$. The convex hull $\langle \rho \rangle$ generated by $\rho$ is a convex body defined by, for $\rho\in C^+(S^{n-1})$,
\begin{align*}
\langle \rho \rangle=\text{conv}\bigg\{\rho(v)v,v\in S^{n-1}\bigg\}.
\end{align*}
Clearly, $[f]^*=\langle \frac{1}{f} \rangle$ and if $\Omega\in\mathcal{K}_o^n$, $\langle \rho_\Omega \rangle=\Omega$.

Let $\Theta\subset S^{n-1}$ be a closed set, $f:\Theta\rightarrow\mathbb{R}$ be continuous, $\delta>0$ and $h_s:\Theta\rightarrow (0,\infty)$ be a continuous function is defined for any $s\in(-\delta,\delta)$ by (see \cite{HY}),
\begin{align*}
\log h_s(v)=\log h(v)+sf(v)+o(s,v),
\end{align*}
for any $v\in \Theta$ and the function $o(s,\cdot):\Theta\rightarrow \mathbb{R}$ is continuous and $\lim_{s\rightarrow 0}o(s,\cdot)/s=0$ uniformly on $\Theta$. Denoted by $[h_s]$ the Wulff shape determined by $h_s$, we shall call $[h_s]$ a logarithmic family of the Wulff shapes formed by $(h,f)$. On occasion, we shall write $[h_s]$ as $[h,f,s]$, and if $h$ happens to be the support function of a convex body $\Omega$ perhaps as $[\Omega,f,s]$, or as  $[\Omega,f,o,s]$, if required for clarity.

Let $g:\Theta\rightarrow \mathbb{R}$ be continuous and $\delta >0$. Let $\rho_s:\Theta\rightarrow (0,\infty)$ be a continuous function defined for each $s\in(-\delta,\delta)$ and each $v\in \Theta$ by
\begin{align*}
\log \rho_s(v)=\log \rho(v)+sg(v)+o(s,v).
\end{align*}
Denoted by $\langle \rho_s\rangle$ the convex hull generated by $\rho_s$, we shall call $\langle \rho_s\rangle$ a logarithmic family of the convex hulls generated by $(\rho,g)$. On occasion $\langle \rho_s\rangle$ as $\langle \rho, g, s\rangle$, and if $\rho$ happens to be the radial function of a convex body $\Omega$ perhaps as $\langle\Omega,g,s\rangle$, or as  $\langle\Omega,g,o,s\rangle$, if required for clarity.

Here we state the following lemma which is required in this paper.
\begin{lemma}\label{lem21} \cite[Lemma 4.2]{HY}
Let $\Theta\subset S^{n-1}$ be a closed set that is not contained in any closed hemisphere of $S^{n-1}$, $\rho_0:\Theta\rightarrow(0,\infty)$ and $g:\Theta\rightarrow\mathbb{R}$ be continuous. If $\langle \rho_s \rangle$ is a logarithmic family of convex hulls of $(\rho_0,g)$, then for $p\in\mathbb{R}$,
\begin{align*}
\lim_{s\rightarrow 0}\frac{ h^{-p}_{\langle \rho_s \rangle}(v)- h^{-p}_{\langle \rho_0 \rangle}(v)}{s}=-ph^{-p}_{\langle \rho_0 \rangle}(v)g(\alpha^*_{\langle \rho_0 \rangle}(v)),
\end{align*}
for all $v\in S^{n-1}\setminus \eta_{\langle \rho_0 \rangle}$. Moreover, there exist $\delta_0>0$ and $M>0$ so that
\begin{align*}
|h^{-p}_{\langle \rho_s \rangle}(v)-h^{-p}_{\langle \rho_0 \rangle}(v)|\leq M|s|,
\end{align*}
for all $v\in S^{n-1}$ and all $s\in(-\delta_0,\delta_0)$.
\end{lemma}

\subsection{Symmetric functions and Hessian operators}(see \cite{BR}) For $k\in \{1,\cdots,n\}$, the $k$-th elementary symmetric function of $A$ is
\begin{align*}
S_k(A)=S(\lambda_1,\cdots,\lambda_n)=\sum_{1\leq i_1<\cdots<i_k\leq n}\lambda_{i_1}\cdots\lambda_{i_k},
\end{align*}
where $A=(a_{ij})$ is a matrix in the space $\mathcal{S}_n$ of the real symmetric $n\times n$ matrices and $\lambda_1,\cdots,\lambda_n$ is eigenvalues of $A$. Notice that $S_k(A)$ is just the sum of all $k\times k$ principal minors of $A$. In particularly, $S_1(A)=\rm tr A$ is the trace of $A$ and $S_n(A)=\det(A)$ is its determinant.

The operator $S_k^{\frac{1}{k}}$, for $k\in \{1,\cdots,n\}$ is homogeneous of degree $1$ and it is increasing and concave if restricted to
\begin{align*}
\Gamma_k=\{A\in\mathcal{S}_n: S_i(A)> 0~~\text{for}~~ i=1,\cdots,k\}.
\end{align*}

Denoting by
\begin{align*}
S_k^{ij}(A)=\frac{\partial}{\partial a_{ij}}S_k(A),
\end{align*}
Euler identity for homogeneous functions gives
\begin{align*}
S_k(A)=\frac{1}{k}S_k^{ij}(A)a_{ij}.
\end{align*}

Let $\Omega$ be an open subset of $\mathbb{R}^n$ and let $u\in C^2(\Omega)$, the $k$-Hessian operator $S_k(D^2u)$ is defined as the $k$-th elementary symmetric function of $D^2u$. Note that
\begin{align*}
S_1(D^2u)=\Delta u \quad \text{and} \quad S_n(D^2u)=\det(D^2 u).
\end{align*}
For $k>1$, the $k$-Hessian operators are fully nonlinear and, in general, not elliptic, unless restricted to the class of $k$-convex functions:
\begin{align*}
\Phi_k^2(\Omega)=\{u\in C^2(\Omega):S_i(D^2u)\geq 0~\text{in}~\Omega, i=1,2,\cdots,k\}.
\end{align*}
Notice that $\Phi^2_n(\Omega)$ coincides with class of $C^2(\Omega)$ convex functions.

A direct computation yields that $(S_k^{1j}(D^2u),\cdots,S_k^{nj}(D^2(u))$ is divergence free (see \cite{RE}), namely,
\begin{align*}
\frac{\partial}{\partial x_i}S_k^{ij}=0.
\end{align*}
Hence $S_k(D^2u)$ can be written in divergence form
\begin{align*}
S_k(D^2u)=\frac{1}{k}S_k^{ij}(D^2u)u_{ij}=\frac{1}{k}(S_k^{ij}(D^2u)u_j)_i,
\end{align*}
where subscripts $i,j$ stand for partial differentiations. For example, when $k=1$, we have $S_1^{ij}=\delta_{ij}$ and $S_1(D^2u)=\delta_{ij}u_{ij}$.

Let $\Omega$ be a bounded connected domain of $\mathbb{R}^n$ of class $C^2$ having principal curvatures $\kappa=(\kappa_1,\cdots,\kappa_{n-1})$ and outer unit normal $v_x$. For $k=1,\cdots,n-1$, we define the $k$-th curvature of $\partial \Omega$ by
\begin{align*}
\sigma_k(\partial \Omega)=\sigma_k(\kappa_1,\cdots,\kappa_{n-1}).
\end{align*}
Moreover, we set
\begin{align*}
\sigma_0=S_0\equiv 1, \quad \sigma_n\equiv 0.
\end{align*}
For example, $\sigma_1$ is equal to $(n-1)$-time the mean curvature of $\partial \Omega$, while $\sigma_{n-1}$ is the Gauss curvature of $\partial \Omega$.

In analogy with the case of functions, $\Omega$ is said $k$-convex, with $ k\in \{1,\cdots,{n-1}\}$, if $\sigma_j\geq 0$ for $j = 1,\cdots,k$ at every point $y=\partial \Omega$. We recall here that any sublevel set of a k-convex function is $(k-1)$-convex (see \cite{CA}).

In general, for $1\leq k\leq n$, a straightforward calculation yields
\begin{align*}
S_k(D^2u)=\sigma_k|Du|^k+\frac{S_ku_iu_lu_{lj}}{|Du|^2}.
\end{align*}
In addition, the following pointwise identity holds (see \cite{RE})
\begin{align*}
\sigma_{k-1}=\frac{S_k^{ij}(D^2u)u_iu_j}{|Du|^{k+1}}.
\end{align*}

\section{\bf The $p$-th dual $k$-torsional measure and variational formula}\label{sec3}

Firstly, we state the following variational formula for the $k$-torsional rigidity was proved in \cite{ZX2}.
\begin{lemma}\label{lem31}\cite[Lemma 3.1]{ZX2} Let $\Omega$ and $\Omega^\prime$ be two convex domains of $C^2_+$, and $h$ and $\theta$ be support functions of $\Omega$ and $\Omega^\prime$, respectively. Let $\Omega_s=\Omega+s\Omega^\prime$ with support function $h_s=h+s\theta$. Suppose $u(X,t)$ is the solution to (\ref{eq101}) in $\Omega_t$. Then
\begin{align*}
\frac{d}{dt}\widetilde{T}_k(\Omega_s)\bigg|_{s=0}=&\int_{S^{n-1}}\theta(x)|Du(X(x))|^{k+1}\sigma_{n-k}(h_{ij}(x)+h(x)\delta_{ij})dx\\
\nonumber=&\int_{S^{n-1}}\theta(x)d\mu_k^{tor}(\Omega,x).
\end{align*}
Here $\mu_k^{tor}(\Omega,\cdot)$ is the $k$-torsional measure of $\Omega$ \cite{ZX2}. Obviously, if $f\in C(S^{n-1})$, then
\begin{align}\label{eq301}
\int_{S^{n-1}}f(v)d\mu^{tor}_k(\Omega,v)=\int_{S^{n-1}}f(\alpha_\Omega(v))H(v)^{k+1}dv,
\end{align}
thus from (\ref{eq104}), we obtain
\begin{align*}
\widetilde{T}_k(\Omega)=\frac{1}{k(n+2)}\int_{S^{n-1}}\rho^{n+1-k}_{\Omega}(v)|Du|^{k+1}dv,
\end{align*}
where $H(v)=|D u(r_\Omega(v))|J(v)^{\frac{1}{k+1}}$ and $J(v)=\frac{\rho_\Omega(v)^{n+1-k}}{h_\Omega(\alpha_\Omega(v))}$, $r_\Omega(v)=\rho_\Omega(v)v$ and $v\in S^{n-1}$.
\end{lemma}

\begin{proposition}\label{pro32}Let $\rho_0:S^{n-1}\rightarrow \mathbb{R}$ and $g:S^{n-1}\rightarrow \mathbb{R}$ be continuous. If $\langle \rho_s\rangle$ is a logarithmic family of convex hulls of $(\rho_0,g)$, then for $1\leq k\leq n-1$,
\begin{align*}
\lim_{s\rightarrow 0}\frac{\widetilde{T}_k(\langle\rho_s\rangle)-\widetilde{T}_k(\langle\rho_0\rangle)}{s}=\int_{S^{n-1}}g(v)\rho^{n+1-k}_{\langle\rho_0\rangle}(v)|Du(r_{\langle\rho_0\rangle}(v))|^{k+1}dv.
\end{align*}
\end{proposition}
\begin{proof}
Using the dominated convergence theorem,  Lemma \ref{lem31}, Lemma \ref{lem21} and (\ref{eq301}), we get
\begin{align*}
&\lim_{s\rightarrow 0}\frac{\widetilde{T}_k(\langle\rho_s\rangle)-\widetilde{T}_k(\langle\rho_0\rangle)}{s}\\
=&\int_{S^{n-1}}\lim_{s\rightarrow 0}\frac{h_{\langle \rho_s\rangle}(\xi)-h_{\langle \rho_0\rangle}(\xi)}{s}d\mu_k^{tor}(\langle \rho_0\rangle,\xi)\\
=&\int_{S^{n-1}}g(\alpha^*_{\langle \rho_0\rangle}(\xi))h_{\langle \rho_0\rangle}(\xi)d\mu_k^{tor}(\langle \rho_0\rangle,\xi)\\
=&\int_{S^{n-1}}g(v)h_{\langle \rho_0\rangle}(\alpha_{\langle \rho_0\rangle}(v))H(v)^{k+1}dv\\
=&\int_{S^{n-1}}g(v)\rho_{\langle \rho_0\rangle}^{n+1-k}(v)|Du(r_{\langle\rho_0\rangle}(v))|^{k+1}dv.
\end{align*}
\end{proof}
\begin{corollary}\label{cor33}Let $\Omega_1, \Omega_2\in\mathcal{K}_o^n$ and $p\in\mathbb{R}$. Then when $p\neq 0$,
\begin{align*}
\lim_{s\rightarrow 0}\frac{\widetilde{T}_k(\langle\rho_{\Omega_1\widetilde{+}_ps\cdot\Omega_2}\rangle)-\widetilde{T}_k(\langle\rho_{\Omega_1}\rangle)}{s}=
\frac{1}{p}\int_{S^{n-1}}\rho_{\Omega_2}(v)^p\rho_{\Omega_1}(v)^{n+1-k-p}|D u(r_{\Omega_1}(v))|^{k+1}dv,
\end{align*}
when $p=0$,
\begin{align*}
\lim_{s\rightarrow 0}\frac{\widetilde{T}_k(\langle\rho_{\Omega_1\widetilde{+}_0s\cdot\Omega_2}\rangle)-\widetilde{T}_k(\langle\rho_{\Omega_1}\rangle)}{s}=
\int_{S^{n-1}}\log\rho_{\Omega_2}(v)\rho_{\Omega_1}(v)^{n+1-k}|D u(r_{\Omega_1}(v))|^{k+1}dv.
\end{align*}
\end{corollary}
\begin{proof}
For sufficiently small $s$,
\begin{align*}
\rho_{\Omega_1\widetilde{+}_ps\cdot\Omega_2}=(\rho^p_{\Omega_1}+s\rho^p_{\Omega_2})^{\frac{1}{p}},\quad p\neq 0,\\
\rho_{\Omega_1\widetilde{+}_0s\cdot\Omega_2}=\rho_{\Omega_1}\rho^s_{\Omega_2},\quad p=0.
\end{align*}
Then
\begin{align}\label{eq302}
\begin{split}
\log(\rho_{\Omega_1\widetilde{+}_ps\cdot\Omega_2})=\log\rho_{\Omega_1}+s\frac{\rho^p_{\Omega_2}}{p\rho^p_{\Omega_1}}+o(s,\cdot),\quad p\neq 0,\\
\log\rho_{\Omega_1\widetilde{+}_0s\cdot\Omega_2}=\log\rho_{\Omega_1}+s\log\rho_{\Omega_2},\quad p=0.
\end{split}
\end{align}

Since $\Omega_1, \Omega_2\in\mathcal{K}_o^n$, the logarithmic family of convex hulls $\langle\rho_{\Omega_1\widetilde{+}_ps\cdot\Omega_2}\rangle=\log\rho_{\Omega_1}+s\frac{\rho^p_{\Omega_2}}{p\rho^p_{\Omega_1}}+o(s,\cdot)$ and $\langle\rho_{\Omega_1\widetilde{+}_0s\cdot\Omega_2}\rangle=\log\rho_{\Omega_1}+s\log\rho_{\Omega_2}$. Let $\langle\rho_0\rangle=\Omega_1$ and $g=\frac{\rho^p_{\Omega_2}}{p\rho^p_{\Omega_1}}$ with $p\neq 0$ and $g=\log\rho_{\Omega_2}$ with $p=0$, thus the desired result follows directly from Proposition \ref{pro32} and formula (\ref{eq302}).
\end{proof}

To simplify the definition, we use the normalized power function \cite{LE011}. For $p\in\mathbb{R}$, and $b\in (0,\infty)$, define $b^{\overline{a}}$, by
\begin{align}\label{eqnpf}
b^{\overline{a}}=
\left\{
    \begin{array}{lc}
       \frac{1}{a}b^a,\quad a\neq 0, \\
        \log b\quad a= 0.\\
    \end{array}
\right.
\end{align}

Using the above variational formula for the $k$-torsional rigidity with respect to the $p$-th radial combination, with the help of (\ref{eqnpf}), we can define the $p$-th dual mixed $k$-torsional rigidity follows: Let $1\leq k\leq n-1$, $p\in\mathbb{R}$ and convex bodies $\Omega_1, \Omega_2\in\mathcal{K}_o^n$, the $p$-th dual mixed $k$-torsional rigidity $\widetilde{Q}_{k,p}(\Omega_1,\Omega_2)$ is defined by
\begin{align}\label{eq303}
\widetilde{Q}_{k,p}(\Omega_1,\Omega_2)=\int_{S^{n-1}}\rho_{\Omega_2}(v)^{\bar{p}}\rho_{\Omega_1}(v)^{n+1-k-p}|Du(r_{\Omega_1}(v))|^{k+1}dv.
\end{align}

When $\Omega_1=\Omega_2$, the $p$-th dual mixed $k$-torsional rigidity of $\Omega_1$ will be shown to be the special case as follows:
\begin{align*}
\widetilde{Q}_{k}(\Omega_1)=\widetilde{Q}_{k,p}(\Omega_1,\Omega_1)=\int_{S^{n-1}}\rho_{\Omega_1}(v)^{\bar{p}}\rho_{\Omega_1}(v)^{n+1-k-p}|D u(r_{\Omega_1}(v))|^{k+1}dv.
\end{align*}
when $p\neq 0$, $\widetilde{T}_k(\Omega_1)=\frac{p}{k(n+2)}\widetilde{Q}_{k}(\Omega_1)$.

Let $\Omega_2=\mathcal{B}$ ($\mathcal{B}$ is a unit ball with $\rho_{\mathcal{B}}(v)=1$) and replace $p$ by $n-p$ in (\ref{eq303}) and $p\neq n$, the $p$-th dual $k$-torsional rigidity of $\Omega_1$ is defined by
\begin{align}\label{eq304}
\widetilde{Q}_{k,n-p}(\Omega_1)=\frac{1}{n-p}\int_{S^{n-1}}\rho_{\Omega_1}(v)^{p+1-k}|Du(r_{\Omega_1}(v))|^{k+1}dv.
\end{align}
When $p=n$, we use $\lim_{p\rightarrow n}\widetilde{Q}_{k,n-p}(\Omega_1)$ to define $\widetilde{Q}_{k,0}(\Omega_1)$, then
\begin{align*}
\widetilde{Q}_{k,0}(\Omega_1)=&\lim_{p\rightarrow n}\frac{1}{n-p}\int_{S^{n-1}}\rho_{\Omega_1}(v)^{p+1-k}|Du(r_{\Omega_1}(v))|^{k+1}dv\\
=&\lim_{p\rightarrow n}\frac{1}{n-p}\int_{S^{n-1}}\rho^{n-p}_{\Omega_1}(v)\rho^{2p+1-n-k}_{\Omega_1}(v)|Du(r_{\Omega_1}(v))|^{k+1}dv\\
=&\int_{S^{n-1}}\log\rho_{\Omega_1}(v)\rho_{\Omega_1}(v)^{n+1-k}|Du(r_{\Omega_1}(v))|^{k+1}dv.
\end{align*}

Because of the need, the definition of the $p$-th dual $k$-torsional measure has already been proposed in the introduction. For convenience, use the normalized power function (\ref{eqnpf}), the definition of the $p$-th dual $k$-torsional measure will be restated as follows.
\begin{definition}\label{def34}
Let $p\in \mathbb{R}$, $1\leq k\leq n-1$ and $\Omega\in\mathcal{K}_o^n$, we define the $p$-th dual $k$-torsional measure by
\begin{align*}
\widetilde{Q}_{k,{n-p}}(\Omega,\eta)=&\int_{\alpha_\Omega^*(\eta)}\rho^{\overline{n-p}}_{\Omega}(v)\rho^{2p+1-n-k}_{\Omega}(v)|Du(r_\Omega(v))|^{k+1}dv\\
=&\int_{S^{n-1}}\bm1_{\alpha_\Omega^*(\eta)}\rho^{\overline{n-p}}_{\Omega}(v)\rho^{2p+1-n-k}_{\Omega}(v)|Du(r_\Omega(v))|^{k+1}dv,
\end{align*}
for each Borel set $\eta\subset S^{n-1}$ and $r_\Omega(v)=\rho_\Omega(v)v$.
\end{definition}

Note that, we will not discuss $p=n$ but only $p\neq n$ in the present paper. Next, we give some properties of the $p$-th dual $k$-torsional measure and variational formula for the $p$-th dual $k$-torsional rigidity with $p\neq n$.

\subsection{\bf The $p$-th dual $k$-torsional measure for special classes of convex bodies}
\begin{lemma}\label{lem35}
Let $\Omega\in\mathcal{K}_o^n$, $1\leq k\leq n-1$ and $p\neq n$. For each function $g:S^{n-1}\rightarrow\mathbb{R}$, $\eta\subset S^{n-1}$, then
\begin{align}\label{eq305}
\int_{S^{n-1}}g(\xi)d\widetilde{Q}_{k,n-p}(\Omega,\xi)=\int_{S^{n-1}}g(\alpha_\Omega(v))\rho_{\Omega}(v)^{p+1-k}|Du(r_\Omega(v))|^{k+1}dv.
\end{align}
\end{lemma}
\begin{proof}
The proof of (\ref{eq305}) refers to \cite [Lemma 3.3]{HY}. Assuming $\psi$ is a simple function on $S^{n-1}$ given by
\begin{align*}
\psi=\sum_{i=1}^mc_i\bm1_{\eta_i}
\end{align*}
with $c_i\in\mathbb{R}$ and Borel set $\eta_i\subset S^{n-1}$. By Definition \ref{def34} with $p\neq n$ and \cite[Equation (2.21)]{HY}, we get
\begin{align*}
\int_{S^{n-1}}\psi(\xi)d\widetilde{Q}_{k,{n-p}}(\Omega,\xi)=&\int_{S^{n-1}}\sum_{i=1}^mc_i\bm1_{\eta_i}(\xi)d\widetilde{Q}_{k,{n-p}}(\Omega,\xi)\\
=&\sum_{i=1}^mc_i\widetilde{Q}_{k,{n-p}}(\Omega,\eta_i)\\
=&\frac{1}{n-p}\int_{S^{n-1}}\sum_{i=1}^mc_i\bm1_{\bm\alpha^*_\Omega(\eta_i)}(v)\rho_{\Omega}(v)^{p+1-k}|Du(r_\Omega(v))|^{k+1}dv\\
=&\frac{1}{n-p}\int_{S^{n-1}}\sum_{i=1}^mc_i\bm1_{\eta_i}(\alpha_\Omega(v))\rho_{\Omega}(v)^{p+1-k}|Du(r_\Omega(v))|^{k+1}dv\\
=&\frac{1}{n-p}\int_{S^{n-1}}\sum_{i=1}^m\psi(\alpha_\Omega(v))\rho_{\Omega}(v)^{p+1-k}|Du(r_\Omega(v))|^{k+1}dv.
\end{align*}
Note that we have established (\ref{eq305}) for simple functions, for a bounded Borel $g$, we choose a sequence of simple functions $\psi_k$ that converge to $g$, uniformly. Then $\psi_k\circ \alpha_\Omega$ to $g\circ \alpha_\Omega$ a.e. with respect to the spherical Lebesgue measure. Since $g$ is a Borel function on $S^{n-1}$ and the radial Gauss map $\alpha_\Omega$ is continuous on $S^{n-1}\setminus\eta_\Omega$, the composite function $g\circ \alpha_\Omega$ is a Borel function on $S^{n-1}\setminus\eta_\Omega$. Hence $g$ and $g\circ \alpha_\Omega$ are Lebesgue integrable on $S^{n-1}$ because $g$
is bounded and $\eta_\Omega$ has the Lebesgue measure zero. Taking the limit $k\rightarrow\infty$ establishes (\ref{eq305}).
\end{proof}

We conclude with an observation regarding the $p$-th dual $k$-torsional measures.

Let $P\in\mathcal{K}_o^n$ be a polytope with outer unit normals $v_1,\cdots,v_m$, $\triangle_i$ be the cone that consists of all of the rays emanating from the origin and passing through the facet of $P$ whose outer unit normal is $v_i$. Then recalling that we abbreviate $\bm\alpha_P^*(\{v_i\})$ by $\bm\alpha_P^*(v_i)$, we have
\begin{align}\label{eq306}
\bm\alpha_P^*(v_i)=S^{n-1}\cap \triangle_i.
\end{align}
If $\eta\subset S^{n-1}$ is a Borel set such that $\{v_1,\cdots,v_m\}\cap\eta=\emptyset$, then $\bm\alpha_P^*(\eta)$ has the spherical Lebesgue measure zero. Thus the $p$-th dual $k$-torsional measure $\widetilde{Q}_{k,{n-p}}(P,\cdot)$ is discrete and concentrated on $\{v_1,\cdots,v_m\}$. By Definition \ref{def34} with $p\neq n$ and equality (\ref{eq306}), we have
\begin{align*}
\widetilde{Q}_{k,{n-p}}(P,\cdot)=\sum_{i=1}^mc_i\delta_{v_i},
\end{align*}
where $\delta_{v_i}$ defines the delta measure concentrated at the point $v_i$ on $S^{n-1}$, and
\begin{align*}
c_i=\frac{1}{n-p}\int_{S^{n-1}\cap \triangle_i}\rho_P(v)^{p+1-k}|Du(r_P(v))|^{k+1}dv.
\end{align*}

\subsection{\bf Properties of the $p$-th dual $k$-torsional measure} In this subsection, we get some properties of the $p$-th dual $k$-torsional measure.
\begin{lemma}\label{lem36}
Let $\Omega\in\mathcal{K}_o^n$ and $p\neq n$, then the $p$-th dual $k$-torsional measure $\widetilde{Q}_{k,{n-p}}(\Omega,\cdot)$ is a Borel measure on $S^{n-1}$.
\end{lemma}
\begin{proof}
It is clear that $\widetilde{Q}_{k,{n-p}}(\Omega,\emptyset)=0$. We only need to prove the countable additivity. Namely, given a sequence of disjoint sets $\eta_i\subset S^{n-1}$, $i=1,2,\cdots$, with $\eta_i\cap \eta_j=\emptyset$ for $i\neq j$, the following formula holds:
\begin{align*}
\widetilde{Q}_{k,{n-p}}(\Omega,\cup_{i=1}^\infty\eta_i)=\sum_{i=1}^\infty\widetilde{Q}_{k,{n-p}}(\Omega,\eta_i).
\end{align*}
To this end, it follows from Definition \ref{def34} with $p\neq n$ that for each Borel set $\eta_i\subset S^{n-1}$, one has
\begin{align*}
\widetilde{Q}_{k,{n-p}}(\Omega,\eta_i)=\frac{1}{n-p}\int_{\alpha_\Omega^*(\eta_i)}\rho^{p+1-k}_{\Omega}(v)|Du(r_\Omega(v))|^{k+1}dv.
\end{align*}
By \cite[Lemmas 2.1-2.4]{HY}, the additivity for Lebesgue integral and fact that the spherical measure of $\omega_\Omega$ is zero, one has
\begin{align*}
\widetilde{Q}_{k,{n-p}}(\Omega,\cup_{i=1}^\infty\eta_i)=&\frac{1}{n-p}\int_{\alpha_\Omega^*(\cup_{i=1}^\infty\eta_i)}\rho^{p+1-k}_{\Omega}(v)|D u(r_\Omega(v))|^{k+1}dv\\
=&\frac{1}{n-p}\int_{\cup_{i=1}^\infty\alpha_\Omega^*(\eta_i)}\rho^{p+1-k}_{\Omega}(v)|Du(r_\Omega(v))|^{k+1}dv\\
=&\frac{1}{n-p}\int_{\cup_{i=1}^\infty\alpha_\Omega^*(\eta_i\setminus \omega_\Omega)}\rho^{p+1-k}_{\Omega}(v)|Du(r_\Omega(v))|^{k+1}dv\\
=&\frac{1}{n-p}\sum_{i=1}^\infty\int_{\alpha_\Omega^*(\eta_i\setminus \omega_\Omega)}\rho^{p+1-k}_{\Omega}(v)|Du(r_\Omega(v))|^{k+1}dv\\
=&\frac{1}{n-p}\sum_{i=1}^\infty\int_{\alpha_\Omega^*(\eta_i)}\rho^{p+1-k}_{\Omega}(v)|Du(r_\Omega(v))|^{k+1}dv\\
&-\frac{1}{n-p}\sum_{i=1}^\infty\int_{\alpha_\Omega^*(\omega_\Omega)}\rho^{p+1-k}_{\Omega}(v)|Du(r_\Omega(v))|^{k+1}dv\\
=&\frac{1}{n-p}\sum_{i=1}^\infty\int_{\alpha_\Omega^*(\eta_i)}\rho^{p+1-k}_{\Omega}(v)|Du(r_\Omega(v))|^{k+1}dv\\
=&\sum_{i=1}^\infty\widetilde{Q}_{k,{n-p}}(\Omega,\eta_i).
\end{align*}
The countable additivity holds and hence $\widetilde{Q}_{k,{n-p}}(\Omega,\cdot)$ is a Borel measure.
\end{proof}
\begin{lemma}\label{lem37}
Let $\Omega\in\mathcal{K}_o^n$ and $p\neq n$, then the $p$-th dual $k$-torsional measure $\widetilde{Q}_{k,{n-p}}(\Omega,\cdot)$ is absolutely continuous with respect to the $(n-k)$-th area measure $S_{n-k}(\Omega,\cdot)$.
\end{lemma}
\begin{proof}
Let $\eta\subset S^{n-1}$ be such that $S_{n-k}(\Omega,\eta)=0$, using the Corfton formula, we conclude that
\begin{align*}
\widetilde{Q}_{k,{n-p}}(\Omega,\eta)=&\frac{1}{n-p}\int_{S^{n-1}}\bm1_{\alpha_\Omega^*(\eta)}\rho^{p+1-k}_{\Omega}(v)|Du(r_\Omega(v))|^{k+1}dv\\
=&\frac{1}{n-p}\int_{S^{n-1}}\bm1_{\alpha_\Omega^*(\eta)}\rho^{p-n}_{\Omega}(v)h_{\Omega}(x)|Du(r_\Omega(v))|^{k+1}dS_{n-k}(\Omega,x)=0,
\end{align*}
since we are integrating over a set of measure zero.
\end{proof}

\begin{lemma}\label{lem38}
If $\Omega_i\in\mathcal{K}_o^n$ with $\Omega_i\rightarrow\Omega_0\in\mathcal{K}_o^n$ and $p\neq n$, then $\widetilde{Q}_{k,{n-p}}(\Omega_i,\cdot)\rightarrow\widetilde{Q}_{k,{n-p}}(\Omega_0,\cdot)$, weakly.
\end{lemma}
\begin{proof}
Let $g:S^{n-1}\rightarrow\mathbb{R}$ be continuous. From (\ref{eq305}), we know that
\begin{align*}
\int_{S^{n-1}}g(\xi)d\widetilde{Q}_{k,{n-p}}(\Omega_i,\xi)=\frac{1}{n-p}\int_{S^{n-1}}g(\alpha_{\Omega_i}(v))\rho_{\Omega_i}(v)^{p+1-k}|Du(r_{\Omega_i}(v))|^{k+1}dv,
\end{align*}
for all $i$. The convergence $\Omega_i\rightarrow\Omega_0$ with respect to the Hausdorff metric implies that $\rho(\Omega_i,v)\rightarrow\rho(\Omega_0,v)$ uniformly on $S^{n-1}$. Since $\Omega_i, \Omega_0\in\mathcal{K}_o^n$, there are positive constants $c$ and $C$ such that for all $v\in S^{n-1}$ and all $i=1,2,\cdots$,
\begin{align*}
c\leq \rho(\Omega_i,v),\rho(\Omega_0,v)\leq C.
\end{align*}
For any given continuous function $g:S^{n-1}\rightarrow\mathbb{R}$ that there is a positive constant $I$ such that for any $i=1,2,\cdots$,
\begin{align*}
|g(\alpha_{\Omega_i})\rho^{p+1-k}(\Omega_i,\cdot)|\leq I \quad \text{and} \quad |g(\alpha_{\Omega_0})\rho^{p+1-k}(\Omega_0,\cdot)|\leq I.
\end{align*}
From $\Omega_i\rightarrow \Omega_0$ and continuity of $r_\Omega$, we know that $r(\Omega_i,v)\rightarrow r(\Omega_0,v)$. The continuity of $Du$ from \cite[Theorem 3.1]{WXJ1} on $\Omega_i,\Omega_0\in\mathcal{K}_o^n$ implies
\begin{align*}|D u(r_{\Omega_i}(v))|\leq C_1\quad\text{and}\quad|D u(r_{\Omega_0}(v))|\leq C_1.
\end{align*}
Thus the desired result directly from \cite[Lemma 2.2]{HY} and dominated convergence theorem:
\begin{align*}
&\frac{1}{n-p}\int_{S^{n-1}}g(\alpha_{\Omega_i}(v))\rho_{\Omega_i}(v)^{p+1-k}|D u(r_{\Omega_i}(v))|^{k+1}dv\\
&\rightarrow\frac{1}{n-p}\int_{S^{n-1}}g(\alpha_{\Omega_0}(v))\rho_{\Omega_0}(v)^{p+1-k}|Du(r_{\Omega_0}(v))|^{k+1}dv,
\end{align*}
this implies that $\widetilde{Q}_{k,{n-p}}(\Omega_i,\cdot)\rightarrow\widetilde{Q}_{k,{n-p}}(\Omega_0,\cdot)$, weakly.
\end{proof}
\subsection{Variational formulas for the $p$-th dual $k$-torsional rigidity}
\begin{theorem}\label{the39}
Let $\eta\subset S^{n-1}$ be a closed set not contained in any closed hemisphere of $S^{n-1}$, $\rho_0:\eta\rightarrow (0,\infty)$ and $g:\eta\rightarrow\mathbb{R}$ be continuous. If $\langle\rho_s\rangle$ is a logarithmic family of the convex hulls of $(\rho_0,g)$, then for $p\neq n$ and $1\leq k\leq n-1$,
\begin{align*}
\lim_{s\rightarrow 0}\frac{\widetilde{Q}_{k,{n-p}}(\langle\rho_s\rangle^*)-\widetilde{Q}_{k,{n-p}}(\langle\rho_0\rangle^*)}{s}=-(p+1-k)(k+2)\int_\eta g(\xi)d\widetilde{Q}_{k,{n-p}}(\langle\rho_0\rangle^*,\xi).
\end{align*}
\end{theorem}
\begin{proof}This proof is similar to \cite[Theorem 4.4]{HY}, however, due to the existence of $|Du|$, it is even more difficult than the proof of \cite[Theorem 4.4]{HY}. Here we omit \cite[page 364: lines 1-22]{HY} to only write the calculation parts. From (\ref{eq304}) and Lemma \ref{lem21}, we have
\begin{align*}
&\lim_{s\rightarrow 0}\frac{\widetilde{Q}_{k,{n-p}}(\langle\rho_s\rangle^*)-\widetilde{Q}_{k,{n-p}}(\langle\rho_0\rangle^*)}{s}=\frac{d}{ds}\widetilde{Q}_{k,{n-p}}(\langle\rho_s\rangle^*)\bigg|_{s=0}\\
=&\frac{1}{n-p}\!\int_{S^{n-1}}\!\bigg(\!\frac{d}{ds}\rho^{p+1-k}_{\langle\rho_s\rangle^*}(v)\bigg|_{s=0}\!|D u(r_{\langle\rho_0\rangle^*}(v))|^{k+1}\!+\!\rho^{p+1-k}_{\langle\rho_0\rangle^*}(v)\frac{d}{ds}\!|Du(r_{\langle\rho_s\rangle^*}(v))|^{k+1}\bigg|_{s=0}\bigg)\!dv\\
=&\frac{1}{n-p}\!\int_{S^{n-1}}\!\bigg(\!\frac{d}{ds}\rho^{p+1-k}_{\langle\rho_s\rangle^*}(v)\bigg|_{s=0}\!|D u(r_{\langle\rho_0\rangle^*}(v))|^{k+1}\!+\!\rho^{p+1-k}_{\langle\rho_0\rangle^*}(v)\frac{d}{ds}\!|Du(\rho_{\langle\rho_s\rangle^*}(v)v)|^{k+1}\bigg|_{s=0}\bigg)\!dv\\
=&\frac{1}{n-p}\!\int_{S^{n-1}}\!\bigg(\!\frac{d}{ds}h^{-({p+1-k})}_{\langle\rho_s\rangle}(v)\bigg|_{s=0}\!|D u(r_{\langle\rho_0\rangle^*}(v))|^{k+1}\!+\!\rho^{p+1-k}_{\langle\rho_0\rangle^*}(v)\frac{d}{ds}\!|D u(h^{-1}_{\langle\rho_s\rangle}(v)v)|^{k+1}\bigg|_{s=0}\bigg)\!dv\\
=&\frac{1}{n-p}\int_{S^{n-1}}\bigg(\lim_{s\rightarrow0}\frac{h^{-({p+1-k})}_{\langle\rho_s\rangle}(v)-h^{-({p+1-k})}_{\langle\rho_0\rangle}(v)}{s}|D u(r_{\langle\rho_0\rangle^*}(v))|^{k+1}\\
&+\rho^{p+1-k}_{\langle\rho_0\rangle^*}(v)\frac{d}{ds}|D u(h^{-1}_{\langle\rho_s\rangle}(v)v)|^{k+1}\bigg|_{s=0}\bigg)dv\\
=&\frac{1}{n-p}\int_{S^{n-1}\backslash \eta_0}-(p+1-k)h^{-({p+1-k})}_{\langle\rho_0\rangle}(v)g(\alpha^*_{\langle\rho_0\rangle}(v))|D u(r_{\langle\rho_0\rangle^*}(v))|^{k+1}dv\\
&+\frac{1}{n-p}\int_{S^{n-1}}\rho^{p+1-k}_{\langle\rho_0\rangle^*}(v)\frac{d}{ds}|D u(h^{-1}_{\langle\rho_s\rangle}(v)v)|^{k+1}\bigg|_{s=0}dv\\
=&\frac{1}{n-p}\int_{S^{n-1}\backslash \eta_0}-(p+1-k)\rho^{p+1-k}_{\langle\rho_0\rangle^*}(v)g(\alpha^*_{\langle\rho_0\rangle}(v))|D u(r_{\langle\rho_0\rangle^*}(v))|^{k+1}dv\\
&+\frac{1}{n-p}\int_{S^{n-1}}\rho^{p+1-k}_{\langle\rho_0\rangle^*}(v)\frac{d}{ds}|D u(h^{-1}_{\langle\rho_s\rangle}(v)v)|^{k+1}\bigg|_{s=0}dv.
\end{align*}
Recall that
\begin{align*}|Du(h^{-1}_{\langle\rho_s\rangle}(v)v)|=-Du(h^{-1}_{\langle\rho_s\rangle}(v)v)\cdot v.
\end{align*}
Thus
\begin{align*}
&\frac{d}{ds}|Du(h^{-1}_{\langle\rho_s\rangle}(v)v)|^{k+1}\bigg|_{s=0}\\
=&(k+1)|Du(h^{-1}_{\langle\rho_0\rangle}(v)v)|^k\frac{d}{ds}|Du(h^{-1}_{\langle\rho_s\rangle}(v)v)|\bigg|_{s=0}\\
=&-(k+1)|D u(h^{-1}_{\langle\rho_0\rangle}(v)v)|^k\bigg((D ^2u(h^{-1}_{\langle\rho_0\rangle}(v)v)\frac{d}{ds}(h^{-1}_{\langle\rho_s\rangle}(v)v))\cdot v+(D \dot{u}(h^{-1}_{\langle\rho_0\rangle}(v)v))\cdot v\bigg)\\
=&\!-(k+1)\!|Du\!(h^{-1}_{\langle\rho_0\rangle}(v)v)|^k\!\bigg(\!(D ^2u(h^{-1}_{\langle\rho_0\rangle}(v)v)\![-h^{-1}_{\langle\rho_0\rangle}(v)\!g(\alpha^*_{\langle\rho_0\rangle}(v))]v)\!\cdot v\!+\!(D \dot{u}(h^{-1}_{\langle\rho_0\rangle}(v)v))\!\cdot \!v\!\bigg)\\
=&-(k+1)|D u(r_{\langle\rho_0\rangle^*}(v))|^k\bigg(D ^2u(r_{\langle\rho_0\rangle^*}(v))[-\rho_{\langle\rho_0\rangle^*}(v)g(\alpha^*_{\langle\rho_0\rangle}(v))]+(D \dot{u}(r_{\langle\rho_0\rangle^*}(v)))\cdot v\bigg)\\
=&(k+1)|D u(r_{\langle\rho_0\rangle^*}(v))|^k D ^2u(r_{\langle\rho_0\rangle^*}(v))\rho_{\langle\rho_0\rangle^*}(v)g(\alpha^*_{\langle\rho_0\rangle}(v))\\
&-(k+1)|D u(r_{\langle\rho_0\rangle^*}(v))|^k(D \dot{u}(r_{\langle\rho_0\rangle^*}(v)))\cdot v.\\
\end{align*}

Denote (see \cite{CO} or \cite{HY0})
\begin{align*}
\frac{d}{ds}|D u(r_{\langle\rho_s\rangle^*}(v))|^{k+1}\bigg|_{s=0}=&\frac{d}{ds}|D u(\rho_{\langle\rho_s\rangle^*}(v)v)|^{k+1}\bigg|_{s=0}\\
=&\mathcal{L}(-(p+1-k)g(\alpha^*_{\langle\rho_0\rangle}(v))\rho^{p+1-k}_{\langle\rho_0\rangle^*}(v))\\
=&\mathcal{L}_1(-(p+1-k)g(\alpha^*_{\langle\rho_0\rangle}(v))\rho^{p+1-k}_{\langle\rho_0\rangle^*}(v))\\
&+\mathcal{L}_2(-(p+1-k)g(\alpha^*_{\langle\rho_0\rangle}(v))\rho^{p+1-k}_{\langle\rho_0\rangle^*}(v)) \end{align*}
with
\begin{align*}
&\mathcal{L}_1(-(p+1-k)g(\alpha^*_{\langle\rho_0\rangle}(v))\rho^{p+1-k}_{\langle\rho_0\rangle^*}(v))\\
&=(k+1)|Du(r_{\langle\rho_0\rangle^*}(v))|^kD ^2u(r_{\langle\rho_0\rangle^*}(v))\rho_{\langle\rho_0\rangle^*}(v)g(\alpha^*_{\langle\rho_0\rangle^*}(v)),
\end{align*}
and
\begin{align*}
&\mathcal{L}_2(-(p+1-k)g(\alpha^*_{\langle\rho_0\rangle}(v))\rho^{p+1-k}_{\langle\rho_0\rangle^*}(v))\\
&=-(k+1)|D u(r_{\langle\rho_0\rangle^*}(v))|^k(D \dot{u}(r_{\langle\rho_0\rangle^*}(v)))\cdot v.
\end{align*}
We can see that $\mathcal{L}$ is a self-adjoint operator on $S^{n-1}$, i.e.
\begin{align*}
\int_{S^{n-1}}\varphi^1\mathcal{L}\varphi^2=\int_{S^{n-1}}\varphi^2\mathcal{L}\varphi^1.
\end{align*}
Indeed, $\mathcal{L}_1$ is self-adjoint obviously. In addition, according to the conclusion of \cite[page 69]{HY0}, we know that $\mathcal{L}_2$ is self-adjoint.

By the $(k+1)$-homogeneity of $l(u)=|D u|^{k+1}$, it yields that
\begin{align*}
\mathcal{L}(\rho^{p+1-k}_{\langle\rho_0\rangle^*})=(k+1)|D u|^{k+1}.
\end{align*}
Hence based on the above calculations and Definition \ref{def34} with $p\neq n$, we get
\begin{align*}
&\lim_{s\rightarrow 0}\frac{\widetilde{Q}_{k,{n-p}}(\langle\rho_s\rangle^*)-\widetilde{Q}_{k,{n-p}}(\langle\rho_0\rangle^*)}{s}=\frac{d}{ds}\widetilde{Q}_{k,{n-p}}(\langle\rho_s\rangle^*)\bigg|_{s=0}\\
=&\frac{1}{n-p}\int_{S^{n-1}\backslash \eta_0}\bigg(-(p+1-k)\rho^{p+1-k}_{\langle\rho_0\rangle^*}(v)g(\alpha^*_{\langle\rho_0\rangle^*}(v))|D u(r_{\langle\rho_0\rangle^*}(v))|^{k+1}\\
&+\rho^{p+1-k}_{\langle\rho_0\rangle^*}(v)\mathcal{L}(-(p+1-k)g(\alpha^*_{\langle\rho_0\rangle}(v))\rho^{p+1-k}_{\langle\rho_0\rangle^*}(v))\bigg)dv\\
=&\frac{1}{n-p}\int_{S^{n-1}\backslash \eta_0}\bigg(-(p+1-k)\rho^{p+1-k}_{\langle\rho_0\rangle^*}(v)g(\alpha^*_{\langle\rho_0\rangle^*}(v))|D u(r_{\langle\rho_0\rangle^*}(v))|^{k+1}\\
&-(p+1-k)g(\alpha^*_{\langle\rho_0\rangle}(v))\rho^{p+1-k}_{\langle\rho_0\rangle^*}\mathcal{L}(\rho^{p+1-k}_{\langle\rho_0\rangle^*}(v))\bigg)dv\\
=&\frac{1}{n-p}\int_{S^{n-1}\backslash \eta_0}\bigg(-(p+1-k)\rho^{p+1-k}_{\langle\rho_0\rangle^*}(v)g(\alpha^*_{\langle\rho_0\rangle^*}(v))|D u(r_{\langle\rho_0\rangle^*}(v))|^{k+1}\\
&-(p+1-k)(k+1)g(\alpha^*_{\langle\rho_0\rangle}(v))\rho^{p+1-k}_{\langle\rho_0\rangle^*}|Du(r_{\langle\rho_0\rangle^*}(v))|^{k+1}\bigg)dv\\
=&\frac{-(p+1-k)(k+2)}{n-p}\int_{S^{n-1}\backslash \eta_0}\rho^{p+1-k}_{\langle\rho_0\rangle^*}(v)g(\alpha^*_{\langle\rho_0\rangle^*}(v))|D u(r_{\langle\rho_0\rangle^*}(v))|^{k+1}dv\\
=&\frac{-(p+1-k)(k+2)}{n-p}\int_{S^{n-1}\setminus\eta_0}(\hat{g}\bm1_\eta)(\alpha_{\langle\rho_0\rangle^*}(v))\rho^{p+1-k}_{\langle\rho_0\rangle^*}(v)|D u(r_{\langle\rho_0\rangle^*}(v))|^{k+1}dv\\
=&-(p+1-k)(k+2)\int_{S^{n-1}\setminus\eta_0}(\hat{g}\bm1_\eta)(\xi)d\widetilde{Q}_{k,{n-p}}(\langle\rho_0\rangle^*,\xi)\\
=&-(p+1-k)(k+2)\int_\eta g(\xi)d\widetilde{Q}_{k,{n-p}}(\langle\rho_0\rangle^*,\xi).
\end{align*}
Here $g(\alpha_{\langle\rho_0\rangle^*}(v))=(\hat{g}\bm1_\eta)(\alpha_{\langle\rho_0\rangle^*}(v))$, it has been proven that $g$ can be extended to a continuous function $\hat{g}:S^{n-1}\rightarrow\mathbb{R}$, (see \cite[page 364]{HY}) for all $v\in S^{n-1}\setminus \eta_0$.
\end{proof}
\begin{theorem}\label{the310}
Let $\Omega\in\mathcal{K}^n_o$ and $f:S^{n-1}\rightarrow\mathbb{R}$ be continuous. If $[h_s]$ is a logarithmic family of the Wulff shapes with respect to $(h_\Omega, f)$, then for $p\neq n$ and $1\leq k\leq n-1$,
\begin{align*}
\lim_{s\rightarrow 0}\frac{\widetilde{Q}_{k,{n-p}}([h_s])-\widetilde{Q}_{k,{n-p}}(\Omega)}{s}=(p+1-k)(k+2)\int_{S^{n-1}}f(\xi)d\widetilde{Q}_{k,{n-p}}(\Omega,\xi).
\end{align*}
\end{theorem}
\begin{proof}
From the definition of the $p$-th dual $k$-torsional rigidity (\ref{eq304}) and Theorem \ref{the39}, we attain
\begin{align*}
&\lim_{s\rightarrow 0}\frac{\widetilde{Q}_{k,{n-p}}([h_s])-\widetilde{Q}_{k,{n-p}}(\Omega)}{s}=\frac{d}{ds}\widetilde{Q}_{k,{n-p}}([h_s])\bigg|_{s=0}\\
=&\frac{1}{n-p}\int_{S^{n-1}}\bigg(\frac{d}{ds}\rho^{p+1-k}_{[h_s]}(v)\bigg|_{s=0}|Du(r_\Omega(v))|^{k+1}+\rho^{p+1-k}_\Omega(v)\frac{d}{ds}|D u(r_{[h_s]}(v))|^{k+1}\bigg|_{s=0}\bigg)dv\\
=&\frac{1}{n-p}\!\int_{S^{n-1}}\!\bigg(\!\lim_{s\rightarrow0}\frac{\rho^{p+1-k}_{[h_s]}\!(v)\!-\!\rho^{p+1-k}_\Omega(v)}{s}\!|\!D u(r_{\Omega}\!(v))|^{k+1}\!+\!\rho^{p+1-k}_\Omega(v)\frac{d}{ds}\!|\!Du(r_{\![h_s]}(v))|^{k+1}\bigg|_{s=0}\bigg)\!dv\\
=&\frac{(p+1-k)(k+2)}{n-p}\int_{S^{n-1}}f(\alpha_\Omega(v))\rho^{p+1-k}_\Omega(v)|Du(r_\Omega(v))|^{k+1}dv\\
=&(p+1-k)(k+2)\int_{S^{n-1}}f(\xi)d\widetilde{Q}_{k,{n-p}}(\Omega,\xi).
\end{align*}

Here the last second equality uses Theorem \ref{the39}. For the convenience of readers, we give a simple explanation. The logarithmic family of Wulff shapes $[h_s]$ is defined as the Wulff shape of $h_s$, where $h_s$ is given by
\begin{align*}
\log h_s=\log h_\Omega+sf+o(s,\cdot).
\end{align*}
This and $\frac{1}{h_\Omega}=\rho_\Omega^*$, allow us to define
\begin{align*}
\log \rho_s^*=\log \rho_\Omega^*-sf-o(s,\cdot),
\end{align*}
and $\rho_s^*$ will generate a logarithmic family of convex hull $\langle\Omega^*,-f,-o,s\rangle$. From \cite[Lemma 2.8]{HY}, we know that $\langle\rho_s\rangle^*=[h_s]$ and $\langle\rho_0\rangle^*=[h_0]$, then
\begin{align*}
[\Omega,f,o,s]=\langle\Omega^*,-f,-o,s\rangle^*.
\end{align*}
Thus the desired result follows directly from Theorem \ref{the39}.
\end{proof}
\begin{corollary}\label{cor311}
Let $\Omega_1, \Omega_2\in\mathcal{K}_o^n$, $p\neq n$ and $1\leq k\leq n-1$. Then
\begin{align*}
&\lim_{s\rightarrow 0}\frac{\widetilde{Q}_{k,{n-p}}((1-s)\Omega_1+s\Omega_2)-\widetilde{Q}_{k,{n-p}}(\Omega_1)}{s}\\
=&(p+1-k)(k+2)[\widetilde{Q}(\Omega_1,\Omega_2)-\widetilde{Q}_{k,{n-p}}(\Omega_1)],
\end{align*}
and
\begin{align*}
&\lim_{s\rightarrow 0}\frac{\widetilde{Q}_{k,{n-p}}((1-s)\Omega_1+_0s\Omega_2)-\widetilde{Q}_{k,{n-p}}(\Omega_1)}{s}\\
=&(p+1-k)(k+2)\int_{S^{n-1}}\log\bigg(\frac{h_{\Omega_2}(\xi)}{h_{\Omega_1}(\xi)}\bigg)
d\widetilde{Q}_{k,{n-p}}(\Omega_1,\xi).
\end{align*}
Here $\widetilde{Q}(\Omega_1,\Omega_2)=\int_{S^{n-1}}\frac{h_{\Omega_2}(v)}{h_{\Omega_1}(v)}d\widetilde{Q}_{k,{n-p}}(\Omega_1,\xi)$.
\end{corollary}
\begin{proof}
For sufficiently small $s$, we define $h_s$ by
\begin{align*}
h_s=(1-s)h_{\Omega_1}+sh_{\Omega_2}=h_{\Omega_1}+s(h_{\Omega_2}-h_{\Omega_1}),
\end{align*}
taking the logarithm of both sides of the above equality, we obtain the following form
\begin{align*}
\log h_s=\log h_{\Omega_1}+s\bigg(\frac{h_{\Omega_2}-h_{\Omega_1}}{h_{\Omega_1}}\bigg)+o(s,\cdot).
\end{align*}
From Theorem \ref{the310}, we get
\begin{align*}
&\lim_{s\rightarrow 0}\frac{\widetilde{Q}_{k,{n-p}}((1-s)\Omega_1+s\Omega_2)-\widetilde{Q}_{k,{n-p}}(\Omega_1)}{s}\\
=&(p+1-k)(k+2)\int_{S^{n-1}}\frac{h_{\Omega_2}-h_{\Omega_1}}{h_{\Omega_1}}d\widetilde{Q}_{k,{n-p}}(\Omega_1,\xi)\\
=&(p+1-k)(k+2)\int_{S^{n-1}}\frac{h_{\Omega_2}}{h_{\Omega_1}}d\widetilde{Q}_{k,{n-p}}(\Omega_1,\xi)\!-\!(p+1-k)(k+2)\int_{S^{n-1}}\!d\widetilde{Q}_{k,{n-p}}(\Omega_1,\xi)\\
=&(p+1-k)(k+2)[\widetilde{Q}(\Omega_1,\Omega_2)-\widetilde{Q}_{k,{n-p}}(\Omega_1)].
\end{align*}
Similarly, for sufficiently small $s$, we can also denote $h_s$ by
\begin{align*}
h_s=h_{\Omega_1}^{1-s}h_{\Omega_2}^s=h_{\Omega_1}\bigg(\frac{h_{\Omega_2}}{h_{\Omega_1}}\bigg)^s,
\end{align*}
then
\begin{align*}
\log h_s=\log h_{\Omega_1}+s\log\bigg(\frac{h_{\Omega_2}}{h_{\Omega_1}}\bigg).
\end{align*}
Thus we have following result by Theorem \ref{the310},
\begin{align*}
\lim_{s\rightarrow 0}\frac{\widetilde{Q}_{k,{n-p}}((1-s)\Omega_1+_0s\Omega_2)\!-\!\widetilde{Q}_{k,{n-p}}(\Omega_1)}{s}=&(p\!+\!1\!-\!k)(k\!+\!2)\!\int_{S^{n-1}}\!\log \frac{h_{\Omega_2}}{h_{\Omega_1}}d\widetilde{Q}_{k,n\!-\!p}(\Omega_1,\xi).
\end{align*}
\end{proof}
\begin{corollary}\label{cor312}
Let $\Omega_1, \Omega_2, \Omega_3\in\mathcal{K}_o^n$, $p\neq n$ and $1\leq k\leq n-1$. Then
\begin{align*}
&\lim_{s\rightarrow 0}\frac{\widetilde{Q}_{k,n-p}((1-s)\Omega_1+_0s\Omega_2,\Omega_3)-\widetilde{Q}_{k,n-p}(\Omega_1,\Omega_3)}{s}\\
=&(p+1-k)(k+2)\int_{S^{n-1}}\log \frac{h_{\Omega_2}}{h_{\Omega_1}}d\widetilde{Q}_{k,n-p}(\Omega_1,\Omega_3,\xi),
\end{align*}
where
\begin{align*}
\widetilde{Q}_{k,n-p}(\Omega_1,\Omega_3,\eta)=\frac{1}{n-p}\int_{\alpha^*_{\Omega_1}(\eta)}\rho_{\Omega_1}(\xi)^{n-p}\rho_{\Omega_3}(\xi)^{p+1-k}|D u(r_{\Omega_1}(\xi))|^{k+1}d\xi.
\end{align*}
\end{corollary}
\begin{proof} The result is directly obtained from replacing $p=n-p$ in formula (\ref{eq303}), Definition \ref{def34} and Corollary \ref{cor311}.
\end{proof}

\section{\bf Geometric flow and associated functionals}\label{sec4}
In this subsection, we give the scalar form of geometric flow (\ref{eq107}) and discuss the key geometric features of associated functionals along the flow (\ref{eq107}) to solve the $p$-th dual Minkowski problem for the $k$-torsional rigidity with $1\leq k\leq n-1$.

Taking the scalar product of both sides of the equation and of the initial condition in the flow (\ref{eq107}) by $v$, by means of the definition of support function (\ref{eq201}), we describe flow (\ref{eq107}) with the support function as the following quantity equation.
\begin{align}\label{eq401}
\left\{
\begin{array}{lc}
\frac{\partial h(x,t)}{\partial t}=\frac{h^2}{f(x)}(|\nabla h|^2+h^2)^{\frac{p-n}{2}}|D u(\overline{\nabla}h,t)|^{k+1}\sigma_{n-k}(x,t)
-\eta(t)h(x,t),\\
h(x,0)=h_0(x).\\
\end{array}
\right.
\end{align}
From $\rho^2=|\nabla h|^2+h^2$, we can write (\ref{eq401}) as
\begin{align}\label{eq402}
\left\{
\begin{array}{lc}
\frac{\partial h(x,t)}{\partial t}=\frac{h^2}{f(x)}\rho^{p-n}|D u(\overline{\nabla}h,t)|^{k+1}\sigma_{n-k}(x,t)
-\eta(t)h(x,t),\\
h(x,0)=h_0(x).\\
\end{array}
\right.
\end{align}
Notice that
\begin{align}\label{eq403}
\frac{1}{\rho(v,t)}\frac{\partial\rho(v,t)}{\partial t}=\frac{1}{h(x,t)}\frac{\partial h(x,t)}{\partial t}.
\end{align}
Thus
\begin{align}\label{eq404}\left\{
\begin{array}{lc}
\frac{\partial \rho(v,t)}{\partial t}=\frac{h}{f(x)}\rho^{p+1-n}|D u(\overline{\nabla}h,t)|^{k+1}\sigma_{n-k}(x,t)
-\eta(t)\rho(v,t),\\
\rho(v,0)=\rho_0(v).\\
\end{array}
\right.\end{align}

Firstly, we show that the functional $\widetilde{Q}_{k,{n-p}}(\Omega_t)$ with $p\neq n$ defined as (\ref{eq304}) is non-decreasing along the flow (\ref{eq107}).
\begin{lemma}\label{lem41}
The functional $\widetilde{Q}_{k,{n-p}}(\Omega_t)$ with $p\neq n$ is non-decreasing along the flow (\ref{eq107}). Namely, $\frac{\partial}{\partial t}\widetilde{Q}_{k,{n-p}}(\Omega_t)\geq0$, the equality holds if and only if the support function of $\Omega_t$ satisfies (\ref{eq106}).
\end{lemma}
\begin{proof}
From Theorem \ref{the310}, we know that
\begin{align*}
\frac{d}{dt}\widetilde{Q}_{k,{n-p}}(\Omega_t)=\frac{(p+1-k)(k+2)}{n-p}\int_{S^{n-1}}\rho^{p-k}\frac{\partial \rho(v,t)}{\partial t}|Du|^{k+1}dv.
\end{align*}
Thus from (\ref{eq404}), (\ref{eq108}), $\rho^{n+1-k}dv=h\sigma_{n-k}dx$ and the H\"{o}lder inequality, we obtain
\begin{align*}
&\frac{d}{dt}\widetilde{Q}_{k,{n-p}}(\Omega_t)\\
=&\frac{(p\!+\!1\!-\!k)\!(k\!+
\!2)}{n-p}\!\int_{S^{n-1}}\!\rho^{p-k}\!\bigg(\!\frac{h}{f(x)}\!\rho^{p-n+1}|Du(\overline{\nabla}h,t)|^{k+1}\sigma_{n-k}(x,t)
\!-\eta(t)\rho(v,t)\!\bigg)\!|Du|^{k+1}dv\\
=&\frac{(p+1-k)(k+2)}{n-p}\bigg[\int_{S^{n-1}}\rho^{2p+1-n-k}\frac{h}{f(x)}|Du|^{2(k+1)}\sigma_{n-k}dv\\
&-\frac{\int_{S^{n-1}}\rho(v,t)^{p+1-k}|Du|^{k+1}dv}{\int_{S^{n-1}}f(x)dx}\int_{S^{n-1}}\rho(v,t)^{p+1-k}|Du|^{k+1}dv\bigg]\\
=&\frac{(p+1-k)\!(k+2)}{(n-p)\int_{S^{n-1}}f(x)dx}\!\bigg\{\bigg[\bigg(\int_{S^{n-1}}\!f(x)dx\bigg)^{\frac{1}{2}}\!\bigg(\!\int_{S^{n-1}}\!\rho^{2p+1-n-k}\!\frac{h}{f(x)}
|Du|^{2(k+1)}\sigma_{n-k}dv\!\bigg)^{\frac{1}{2}}\bigg]^2\\
&-\bigg(\int_{S^{n-1}}\rho(v,t)^{p+1-k}|Du|^{k+1}dv\bigg)^2\bigg\}\\
=&\bigg\{\bigg[\bigg(\int_{S^{n-1}}f(x)\frac{\rho^{n+1-k}}{h\sigma_{n-k}}dv\bigg)^{\frac{1}{2}}\bigg(\int_{S^{n-1}}\rho^{2p+1-n-k}\frac{h}{f(x)}
|Du|^{2(k+1)}\sigma_{n-k}dv\bigg)^{\frac{1}{2}}\bigg]^2\\
&-\bigg(\int_{S^{n-1}}\rho(v,t)^{p+1-k}|Du|^{k+1}dv\bigg)^2\bigg\}\frac{(p+1-k)(k+2)}{(n-p)\int_{S^{n-1}}f(x)dx}\\
=&\bigg\{\!\bigg[\!\bigg(\int_{S^{n-1}}\!\bigg[\bigg(f(x)\frac{\rho^{n+1-k}}{h\sigma_{n-k}}\bigg)^{\frac{1}{2}}\bigg]^2\!dv\bigg)^{\frac{1}{2}}\!\bigg(\int_{S^{n-1}}
\!\bigg[\bigg(\rho^{2p+1-n-k}\!\frac{h}{f(x)}|Du|^{2(k+1)}\sigma_{n-k}\!\bigg)^{\frac{1}{2}}\bigg]^2\!dv\!\bigg)^{\frac{1}{2}}\!\bigg]^2\\
&-\bigg(\int_{S^{n-1}}\rho(v,t)^{p+1-k}|Du|^{k+1}dv\bigg)^2\bigg\}\frac{(p+1-k)(k+2)}{(n-p)\int_{S^{n-1}}f(x)dx}\\
\geq&\frac{(p+1-k)(k+2)}{(n-p)\int_{S^{n-1}}f(x)dx}\bigg\{\bigg[\int_{S^{n-1}}\bigg(f(x)\frac{\rho^{n+1-k}}{h\sigma_{n-k}}\bigg)^{\frac{1}{2}}\bigg(\rho^{2p+1-n-k}\frac{h}{f(x)}
|Du|^{2(k+1)}\sigma_{n-k}\bigg)^{\frac{1}{2}}dv\bigg]^2\\
&-\bigg(\int_{S^{n-1}}\rho(v,t)^{p+1-k}|Du|^{k+1}dv\bigg)^2\bigg\}\\
=&\frac{(p\!+\!1\!-\!k)\!(k\!+\!2)}{(n\!-\!p)\!\int_{S^{n-1}}\!f(x)dx}\!\bigg[\!\bigg(\int_{S^{n-1}}\!\rho(v,t)^{p+1-k}\!|Du|^{k+1}\!dv\!\bigg)^2-
\!\bigg(\int_{S^{n-1}}\!\rho(v,t)^{p+1-k}\!|Du|^{k+1}\!dv\!\bigg)^2\bigg]\\
=&0.
\end{align*}

According to the equality condition of H\"{o}lder inequality, we know that the above equality holds if and only if
\begin{align*}
\bigg(f(x)\frac{\rho^{n+1-k}}{h\sigma_{n-k}}\bigg)^{\frac{1}{2}}=\tau\bigg(\rho^{2p+1-n-k}\frac{h}{f(x)}
|Du|^{2(k+1)}\sigma_{n-k}\bigg)^{\frac{1}{2}},
\end{align*}
the above equation can be simplified as
\begin{align*}
f(x)=\tau\rho^{p-n}h{f(x)}|Du|^{k+1}\sigma_{n-k},
\end{align*}
namely,
\begin{align*}
f(x)=\tau(h^2+|\nabla h|^2)^{\frac{p-n}{2}}h{f(x)}|Du|^{k+1}\sigma_{n-k}.
\end{align*}
This is equation (\ref{eq106}) with $\tau=\frac{1}{\eta(t)}$.
\end{proof}

Moreover, we prove the functional (\ref{eq109}) is unchanged along the flow (\ref{eq107}). Please refer to the following lemma for details.
\begin{lemma}\label{lem42} The functional (\ref{eq109}) is unchanged along the flow (\ref{eq107}). That is $\frac{d}{d t}\Phi(\Omega_t)=0$.
\end{lemma}
\begin{proof}
By (\ref{eq109}), (\ref{eq108}), (\ref{eq402}) and $\rho^{n+1-k}dv=h\sigma_{n-k}dx$, we obtain the following result,
\begin{align*}
\frac{\partial}{\partial t}\Phi(\Omega_t)=&\int_{S^{n-1}}\frac{f(x)}{h}\frac{\partial h}{\partial t}dx\\
=&\int_{S^{n-1}}\frac{f(x)}{h}\bigg(\frac{h^2}{f(x)}\rho^{p-n}|D u|^{k+1}\sigma_{n-k}(x,t)
-\eta(t)h(x,t)\bigg)dx\\
=&\int_{S^{n-1}}\!h\rho^{p-n}|D u|^{k+1}\sigma_{n-k}(x,t)dx\!-\!\frac{\int_{S^{n-1}}\rho^{p+1-k}|Du|^{k+1}dv}{\int_{S^{n-1}}f(x)dx}\!\int_{S^{n-1}}f(x)dx\\
=&\int_{S^{n-1}}\rho^{p-n}|D u|^{k+1}\rho^{n+1-k}dv-\int_{S^{n-1}}\rho^{p+1-k}|Du|^{k+1}dv\\
=&0.
\end{align*}
\end{proof}

\section{\bf Priori estimates}\label{sec5}

In this subsection, we establish the $C^0$, $C^1$ and $C^2$ estimates for solutions to equation (\ref{eq402}). In the following of this paper, we always assume that $\partial\Omega_0$ be a smooth, closed and strictly convex hypersurface in $\mathbb{R}^n$ containing the origin in its interior. $h:S^{n-1}\times [0,T)\rightarrow \mathbb{R}$ is a smooth solution to equation (\ref{eq402}) with the initial $h(\cdot,0)$ the support function of $\partial\Omega_0$. Here $T$ is the maximal time for existence of the smooth solutions to equation (\ref{eq402}).

\subsection{$C^0$, $C^1$ estimates}

In order to complete the $C^0$ estimate, we firstly need to introduce the following lemma for convex bodies.
\begin{lemma}\label{lem51}\cite[Lemma 2.6]{CH}
Let $\Omega\in\mathcal{K}^n_o$, $h$ and $\rho$ are respectively the support function and the radial function of $\Omega$, and $x_{\max}$ and $\xi_{\min}$ are two points such that
$h(x_{\max})=\max_{S^{n-1}}h$ and $\rho(\xi_{\min})=\min_{S^{n-1}}\rho$. Then
\begin{align*}
\max_{S^{n-1}}h=&\max_{S^{n-1}}\rho \quad \text{and} \quad \min_{S^{n-1}}h=\min_{S^{n-1}}\rho;\end{align*}
\begin{align*}h(x)\geq& x\cdot x_{\max}h(x_{\max}),\quad \forall x\in S^{n-1};\end{align*}
\begin{align*}\rho(\xi)\xi\cdot\xi_{\min}\geq&\rho(\xi_{\min}),\quad \forall \xi\in S^{n-1}.\end{align*}
\end{lemma}

\begin{lemma}\label{lem52}
Let $\Omega_t$ be a smooth strictly convex solution to the flow (\ref{eq107}) in $\mathbb{R}^n$ and $u(X,t)$ be the smooth admissible solution of (\ref{eq101}) in $\Omega_t$, and $f$ be a positive smooth function on $S^{n-1}$. Then there is a positive constant $C$ independent of $t$ such that
\begin{align}\label{eq501}
\frac{1}{C}\leq h(x,t)\leq C, \ \ \forall(x,t)\in S^{n-1}\times[0,T),
\end{align}
\begin{align}\label{eq502}
\frac{1}{C}\leq \rho(v,t)\leq C, \ \ \forall(v,t)\in S^{n-1}\times[0,T).
\end{align}
Here $h(x,t)$ and $\rho(v,t)$ are the support function and the radial function of $\Omega_t$, respectively.
\end{lemma}
\begin{proof}
Due to  $\rho(v,t)v=\nabla h(x,t)+h(x,t)x$. Clearly, one sees
\begin{align*}
\min_{S^{n-1}} h(x,t)\leq \rho (v,t)\leq \max_{S^{n-1}} h(x,t).
\end{align*}
This implies that estimate (\ref{eq501}) is equivalent to estimate (\ref{eq502}). Thus we only need to estimate (\ref{eq501}) or (\ref{eq502}).

Here, we construct a concave function $\mathcal{F}(y)=-\frac{2}{n-1}e^{(n-1)y}-ye^{-(n-1)y}(n\geq 2), y\in\mathbb{R}$. Now, we introduce a new auxiliary function
\begin{align*}
\Phi_1(\Omega_t)=c+\mathcal{F}\bigg(\int_{S^{n-1}}\log h(x,t)f(x)dx\bigg),
\end{align*} 
where $c$ is a positive constant (independent of $t$) such that 
\begin{align*}
\Phi_1(\Omega_0)=c+\mathcal{F}\bigg(\int_{S^{n-1}}\log h(x,0)f(x)dx\bigg)=c_0>0.
\end{align*}
Then
\begin{align*}
\frac{\partial}{\partial t}\Phi_1(\Omega_t)
=\mathcal{F}^{\prime}\bigg(\int_{S^{n-1}}\log h(x,t)f(x)dx\bigg)\int_{S^{n-1}}\frac{f(x)}{h}\frac{\partial h}{\partial t}dx.
\end{align*}
From the proof of Lemma \ref{lem42}, we know that $\Phi_1(\Omega_t)$ is also unchanged along the flow (\ref{eq107}) with $t\in[0,T)$. We now denote $h_{\min_t}(x)$ (or $h_{\inf_t}(x)$) by the minimum value or infimum of $h(x,t)$ $w.r.t.$ time $t$. Here, for convenience and without loss of generality, we assume $\int_{S^{n-1}}f(x)dx=1$, since $\Phi_1(\Omega_t)$ is unchanged along the flow (\ref{eq107}) with $t\in[0,T)$ and $\mathcal{F}$ is a concave function, then by Jensen inequality, there is
\begin{align*}
\Phi_1(\Omega_0)=\Phi_1(\Omega_t)=&c+\mathcal{F}\bigg(\int_{S^{n-1}}\log h(x,t)f(x)dx\bigg)=c+\mathcal{F}\bigg(\int_{S^{n-1}}\log h_{\min_t}(x)f(x)dx\bigg)\\
\geq &c+\int_{S^{n-1}}\mathcal{F}[\log h_{\min_t}(x)]f(x)dx.
\end{align*}
Then 
\begin{align*}
\Phi_1(\Omega_0)^{-1}\leq&\frac{1}{c+\int_{S^{n-1}}\mathcal{F}[\log h_{\min_t}(x)]dx}.
\end{align*}
The foregoing inequality shows that $\log h_{\min_t}(x)\nrightarrow -\infty$, i,e. $h_{\min_t}(x)\nrightarrow 0$ (or $ h_{\inf_t}(x)\nrightarrow 0$) for any $x\in S^{n-1}$. In fact, if there is a point $x_0\in S^{n-1}$ such that $h_{\min_t}(x)\rightarrow 0$ when $x\rightarrow x_0$ and denote the neighborhood of $x_0$ being with $U(x_0,\epsilon)\subset S^{n-1}$ and the measure $|U(x_0,\epsilon)|>0$ for any small $\varepsilon>0$. Thus we have
\begin{align*}
&\Phi_1(\Omega_0)^{-1}\\
\leq&\frac{1}{c+\min\limits_{S^{n-1}}f(x)\int_{U(x_0,\epsilon)}\mathcal{F}[\log h_{\min_t}(x)]dx+\int_{S^{n-1}\setminus U(x_0,\epsilon)}\mathcal{F}[\log h_{\min_t}(x)]f(x)dx}\\
=&\frac{1}{\!c+\!\min\limits_{S^{n-1}}\!f(x)\!\int_{U(x_0,\epsilon)}\bigg(\frac{\!-\!2h_{\min_t}(x)^{n-1}}{n-1}\!-\frac{\!\log \!h_{\min_t}(x)\!}{h_{\min_t}(x)^{(n-1)}}\bigg)dx+\int_{S^{n-1}\setminus U(x_0,\epsilon)}\mathcal{F}[\log h_{\min_t}(x)]f(x)dx}\\
\rightarrow&\frac{1}{c-\infty+\int_{S^{n-1}\setminus U(x_0,\epsilon)}\mathcal{F}[\log h_{\min_t}(x)]f(x)dx}=0,
\end{align*}
as $x\rightarrow x_0$. This is a contradictory with $\Phi_1(\Omega_0)^{-1}=\frac{1}{c_0}>0$, namely, there is no point that makes $h_{\min_t}(x)\rightarrow 0$. Thus one can take $\delta>0$ independent of $t$ small enough to draw that $h_{\min_t}(x)\geq\delta$. The same discussion applies to $h_{\inf_t}(x)\geq\delta$. The support function for low one-dimensional convex bodies can be similarly proven. At the same time, we can obtain $\Omega_t$ is a convex body containing the origin in its interior point, i.e. $\Omega_t\in\mathcal{K}_o^n$.

Next, we will derive at the uniform upper bound of $h(x,t)$. We have attained $\Omega_t\in\mathcal{K}_o^n$, thus from Lemma \ref{lem51}, there is
\begin{align*}
h(x,t)\geq x\cdot x^t_{\max} h(x^t_{\max},t), \quad \forall x\in S^{n-1},
\end{align*}
where $x^t_{\max}$ is a point such that $h(x^t_{\max},t)=\max_{S^{n-1}}h(\cdot,t)$. Now, from Lemma \ref{lem42}, we obtain
\begin{align*}
\Phi(\Omega_0)=\Phi(\Omega_t)=&\int_{S^{n-1}}\log h(x,t)f(x)dx\\
\geq &\int_{S^{n-1}}f(x)\log[h(x^t_{\max},t)x\cdot x^t_{\max}]dx\\
\geq&\log h(x^t_{\max},t)\!\int_{S^{n-1}}\!f(x)dx+\int_{\{x\in S^{n-1}:\!x\cdot x^t_{\max}\geq\frac{1}{2}\}}\!f(x)\log(x\cdot x^t_{\max})dx\\
\geq&C\log h(x^t_{\max},t)-c\int_{\{x\in S^{n-1}: x\cdot x^t_{\max}\geq\frac{1}{2}\}}f(x)dx\\
\geq&C\log h(x^t_{\max},t)-c_1.
\end{align*}
This yields
\begin{align*}
\sup h(x^t_{\max},t)\leq e^{\frac{\Phi(\Omega_0)+c_1}{C}}.
\end{align*}
Here $C$, $c$ and $c_1$ are positive constants independent of $t$.
\end{proof}

\begin{lemma}\label{lem53}Let $\Omega_t$ be a smooth strictly convex solution to the flow (\ref{eq107}) in $\mathbb{R}^n$ and $u(X,t)$ be the smooth admissible solution of (\ref{eq101}) in $\Omega_t$, and $f$ be a positive smooth function on $S^{n-1}$. Then there is a positive constant $C$ independent of $t$ such that
\begin{align}\label{eq503}|\nabla h(x,t)|\leq C,\quad\forall(x,t)\in S^{n-1}\times [0,T),\end{align}
and
\begin{align}\label{eq504}|\nabla \rho(v,t)|\leq C,\quad \forall(v,t)\in S^{n-1}\times [0,T).\end{align}
\end{lemma}

\begin{proof}
The desired results immediately follow from Lemma \ref{lem52} and the following identities (see e.g. \cite{LR})
\begin{align*}
h=\frac{\rho^2}{\sqrt{\rho^2+|\nabla\rho|^2}},\qquad\rho^2=h^2+|\nabla h|^2.\end{align*}
\end{proof}

\begin{lemma}\label{lem54}Let $\Omega_t$ be a smooth strictly convex solution to the flow (\ref{eq107}) in $\mathbb{R}^n$ and $u(X,t)$ be the smooth admissible solution of (\ref{eq101}) in $\Omega_t$, and $f$ be a positive smooth function on $S^{n-1}$. Then there is a positive constant $C$ independent of $t$ such that
\begin{align*}\frac{1}{C}\leq\eta(t)\leq C.\end{align*}
\end{lemma}

\begin{proof}
From the definition of $\eta(t)$ and Lemma \ref{lem41}, we can directly obtain the lower bound of $\eta(t)$, namely,
\begin{align*}
\eta(t)=\frac{\int_{S^{n-1}}\rho^{p+1-k}|D u(X,t)|^{k+1}dv}{\int_{S^{n-1}}f(x)dx}=\frac{(n-p)\widetilde{Q}_{k,{n-p}}(\Omega_t)}{\int_{S^{n-1}} f(x)dx}\geq\frac{(n-p)\widetilde{Q}_{k,{n-p}}(\Omega_0)}{\int_{S^{n-1}}f(x)dx}.
\end{align*}

Since $\Omega_t$ is a smooth strictly convex solution to the flow (\ref{eq107}) for any $t\in[0,T)$ and we have obtained uniform upper bound and uniform lower bound of $\Omega_t$ in Lemma \ref{lem52}. Thus there exist the balls $B_R$ and $B_r$ with radii of $R\leq R_0<\infty$ and $r\geq \delta>0$ such that $B_r\subset \Omega_t\subset B_R$, for the balls $B_R$ and $B_r$, we have for any $x\in S^{n-1}$,
\begin{align*}
\left\{
\begin{array}{lc}
S_k(D^2u_R(X(x)))=1\ \ \  \text{in} \ \ \ B_R,\\
u_R=0,\ \ \ \ \ \ \ \ \text{on} \ \ \ \ \partial B_R,\\
\end{array}
\right.
\end{align*}
and
\begin{align*}
\left\{
\begin{array}{lc}
S_k(D^2u_r(X(x)))=1\ \ \  \text{in} \ \ \ B_r,\\
u_r=0,\ \ \ \ \ \ \ \ \text{on} \ \ \ \ \partial B_r.\\
\end{array}
\right.
\end{align*}
The analysis of radial symmetric solutions provides an expression for the explicit solution \cite{NMA}, for example, for ball $B_R$, $u_R=c_{n,k}(R^2-|X(x)|^2)$, where $c_{n,k}$ depends on dimension $n$ and $k$, then $|Du_R|=2c_{n,k}R$. Similarly, $|Du_r|=2c_{n,k}r$.

Since $B_r\subset \Omega_t\subset B_R$ and $u=0$ on $\partial\Omega_t$, moreover, $u$ is a smooth admissible solution of (\ref{eq101}) on $\Omega_t$. For any $x\in S^{n-1}$ and $t\in[0,T)$, any point $X(x)\in\partial\Omega_t$, there exists ball $B_r$ such that $B_r\subset\Omega_t$ and $\partial\Omega_t\cap B_r=X(x)$. Because of the same equation and $u(\cdot,t)\geq 0=u_r(\cdot)$ on $\partial B_r$, hence using the maximum principle of $k$-Hessian equation \cite{CA}, we can obtain $u(\cdot,t)\geq u_r(\cdot)$ in $B_r$ and $u(X(x,t),t)=u_r(X(x))$, we have $|Du(X(x,t),t)|\geq|Du_r(X(x))|$. Similarly, we attain the upper bound $|Du(X(x,t),t)|\leq |Du_R(X(x))|$ by comparing it with $u_R$. Thus we obtain
\begin{align*}
c\delta\leq |Du(X(x,t),t)|\leq CR_0,
\end{align*}
where $c$ and $C$ independent of $t$.

The upper bound of $\eta(t)$ can be immediately obtained by upper bound of $|Du(X(x,t),t)|$ and $\rho(v,t)$.
\end{proof}

\subsection{$C^2$ estimate}
In this subsection, we establish the upper bound and the lower bound of principal curvature. This will show that equation (\ref{eq402}) is uniformly parabolic. Firstly, we establish the lower bound of $\sigma_{n-k}(x,t)$.
\begin{lemma}\label{lem55} Let $p< n-2$, under the conditions of Lemma \ref{lem52}, then there is a positive constant $C_0$ independent of $t$ such that
\begin{align*}
\sigma_{n-k}\geq C_0.
\end{align*}
\end{lemma}

\begin{proof}Combining the auxiliary function in \cite{KR0}, we construct an auxiliary function that conforms to the curvature flow (\ref{eq107}) as follows:
\begin{align*}
E=&\log\bigg(\frac{h^2}{f(x)}\rho^{p-n}|Du|^{k+1}\sigma_{n-k}\bigg)-A\frac{\rho^2}{2},
\end{align*}
where $A$ is a positive constant which will be chosen later.

Denote $\frac{h^2}{f(x)}\rho^{p-n}|Du|^{k+1}\sigma_{n-k}=G\sigma_{n-k}=F$ and $\frac{\partial h(x,t)}{\partial t}=h_t$, then $h_t=F-\eta(t)h$ and $\rho_t=\frac{\rho}{h}(F-\eta(t)h)$ by (\ref{eq402}) and (\ref{eq403}). Thus the evolution equation of $E$ is written as
\begin{align*}
\frac{\partial E}{\partial t}=&\frac{1}{F}\frac{\partial F}{\partial t}-A\frac{\partial(\frac{\rho^2}{2})}{\partial t}.
\end{align*}
Now, we compute the evolution equation of $F$,
\begin{align*}
\frac{\partial F}{\partial t}=\sigma_{n-k}\frac{\partial G}{\partial t}+G\frac{\partial\sigma_{n-k}}{\partial t},
\end{align*}
where
\begin{align*}
\frac{\partial G}{\partial t}=\frac{1}{f(x)}\bigg(2h\rho^{p-n}|Du|^{k+1}h_t+(p-n)h^2\rho^{p-n-1}|Du|^{k+1}\rho_t+(k+1)h^2\rho^{p-n}|Du|^k\frac{\partial |Du|}{\partial t}\bigg).
\end{align*}
Since $|Du(X(x,t),t)|=-\langle Du(X(x,t),t),x\rangle$, $X(x,t)=h_ie_i+hx$, then
\begin{align}\label{eq505}
\frac{\partial |Du(X(x,t),t)|}{\partial t}=-[\langle (D^2u)x,(h_{ti}e_i+h_tx)\rangle+\langle D\dot{u},x\rangle].
\end{align}
From $u(X(x,t),t)=0$ on $\partial\Omega_t$, taking the derivative of both sides with respect to $t$, then we obtain
\begin{align*}
\dot{u}+Du \frac{\partial X(x,t)}{\partial t}=0,
\end{align*}
thus
\begin{align}\label{eq506}
\dot{u}=-Du\cdot (h_{ti}e_i+h_tx)=|Du|x\cdot (h_{ti}e_i+h_tx)=|Du|h_t(x).
\end{align}
From (\ref{eq506}), we further calculate
\begin{align}\label{eq507}
\langle D\dot{u},x\rangle=&\langle D(|Du|h_t),x\rangle=(\langle |Du|^{-1}Du D^2u,x\rangle)h_t+\langle|Du|(\nabla h)_t,x\rangle\\
\nonumber=&(\langle |Du|^{-1}Du D^2u,x\rangle)h_t+\langle|Du|(h_ie_i+hx)_t,x\rangle\\
\nonumber=&(\langle |Du|^{-1}Du D^2u,x\rangle)h_t+|Du|h_t.
\end{align}
Substituting (\ref{eq507}) into (\ref{eq505}), we obtain
\begin{align}\label{eq508}
&\frac{\partial |Du(X(x,t),t)|}{\partial t}\\
\nonumber=&-h_{ti}\langle (D^2u)x,e_i\rangle-h_t \langle(D^2u)x,x\rangle -(\langle |Du|^{-1}Du D^2u,x\rangle)h_t-|Du|h_t\\
\nonumber=&-(h_t)_i\langle (D^2u)x,e_i\rangle-\bigg(\langle(D^2u)x,x\rangle +(\langle |Du|^{-1}Du D^2u,x\rangle)+|Du|\bigg)h_t.
\end{align}
Thus combining (\ref{eq508}), we obtain
\begin{align}\label{eq509}
&\frac{\partial G}{\partial t}\sigma_{n-k}\\
\nonumber=&\frac{\sigma_{n-k}}{f(x)}\bigg\{2h\rho^{p-n}|Du|^{k+1}h_t+(p-n)h^2\rho^{p-n-1}|Du|^{k+1}\rho_t\\
\nonumber&+\!(k\!+\!1)h^2\!\rho^{p-n}\!|Du|^k\!\bigg[\!-\!(h_t)_k\langle (D^2u)x,e_k\rangle\!-\!\bigg(\!\langle(\!D^2u)x,x\rangle\!+\!(\!\langle |Du|^{-1}Du D^2u,x\rangle\!)\!+\!|Du|\!\bigg)h_t\!\bigg]\!\bigg\}.
\end{align}

Recall $\sigma_{n-k}(x,t)=\sigma_{n-k}(\omega_{ij}(x,t))$, $\omega_{ij}(x,t)=h_{ij}(x,t)+h(x,t)\delta_{ij}$ and $d_{ij}=\frac{\partial \sigma_{n-k}}{\partial \omega_{ij}}$, then
\begin{align}\label{eq510}
\frac{\partial \sigma_{n-k}}{\partial t}=&d_{ij}\frac{\partial \omega_{ij}}{\partial t}=d_{ij}\nabla_{ij}(h_t)+d_{ij}h_t\delta_{ij}\\
\nonumber=&d_{ij}\nabla_{ij}(F-\eta(t)h)+d_{ij}\delta_{ij}(F-\eta(t)h)\\
\nonumber=&d_{ij}F_{ij}-\eta(t)d_{ij}h_{ij}+d_{ij}\delta_{ij}(F-\eta(t)h)\\
\nonumber=&d_{ij}F_{ij}+Fd_{ij}\delta_{ij}-\eta(t)d_{ij}(h_{ij}+h\delta_{ij})\\
\nonumber=&d_{ij}F_{ij}+Fd_{ij}\delta_{ij}-(n-k)\eta(t)\sigma_{n-k},
\end{align}
where we use the $(n-k)$-degree homogeneity of $\sigma_{n-k}$ in the last equality and obtain $d_{ij}\omega_{ij}=(n-k)\sigma_{n-k}$.

We know that $\rho^2=h^2+|\nabla h|^2$, thus
\begin{align}\label{eq511}
\frac{\partial(\frac{\rho^2}{2})}{\partial t}=&\frac{1}{2}\frac{\partial(h^2+|\nabla h|^2)}{\partial t}=hh_t+\sum h_ih_{it}\\
\nonumber=&h(F-\eta(t)h)+\sum h_i(F_i-\eta(t)h_i)=hF+\sum h_iF_i-\eta(t)\rho^2.
\end{align}

Combining (\ref{eq509}), (\ref{eq510}) and (\ref{eq511}), we get
\begin{align*}
\frac{\partial E}{\partial t}=&\frac{1}{F}\bigg\{\frac{\sigma_{n-k}}{f(x)}\bigg[2h\rho^{n-p}|Du|^{k+1}h_t+(p-n)h^2\rho^{p-n-1}|Du|^k\rho_t\\
&\!+\!(\!k\!+\!1\!)h^2\!\rho^{p-n}\!|Du|^k\!\bigg[\!-\!(h_t)_i\langle (D^2u)x,e_i\rangle\!-\!\bigg(\!\langle(\!D^2u)x,x\rangle\!+\!(\!\langle |Du|^{-1}Du D^2u,x\rangle\!)\!+\!|Du|\!\bigg)h_t\!\bigg]\\
&+G\bigg(d_{ij}F_{ij}+Fd_{ij}\delta_{ij}-(n-k)\eta(t)\sigma_{n-k}\bigg)\bigg\}-A[hF+\sum h_iF_i-\eta(t)\rho^2]\\
=&\frac{\sigma_{n-k}}{f(x)}\bigg[\frac{2h\rho^{p-n}|Du|^{k+1}(F-\eta(t)h)}{F}+\frac{(p-n)h\rho^{p-n}|Du|^{k+1}(F-\eta(t)h)}{F}\\
&-\frac{(k+1)h^2\rho^{p-n}|Du|^k((F-\eta(t)h))_i\langle (D^2u)x,e_i\rangle}{F}\\
&-\frac{(k+1)h^2\rho^{p-n}|Du|^k\bigg(\langle(D^2u)x,x\rangle +(\langle |Du|^{-1}Du D^2u,x\rangle)+|Du|\bigg)(F-\eta(t)h)}{F}\bigg]\\
&+\frac{G}{F}\bigg(d_{ij}F_{ij}+Fd_{ij}\delta_{ij}-(n-k)\eta(t)\sigma_{n-k}\bigg)-A[hF+\sum h_iF_i-\eta(t)\rho^2].
\end{align*}
Suppose the spatial minimum of $E$ is attained at a point $(x_0,t)$, then $F_i=0$, $F_{ij}\geq 0$, thus dropping some positive terms and rearranging terms yield
\begin{align*}
\frac{\partial E}{\partial t}\!\geq &\frac{\sigma_{n-k}}{f(x)}\bigg[\frac{(p\!+\!2\!-\!n)h\rho^{p-n}\!|Du|^{k+1}(e^{E+A\frac{\rho^2}{2}}\!-\!\eta(t)h)}{e^{E+A\frac{\rho^2}{2}}}\!
+\!\frac{(k\!+\!1)h^2\!\rho^{p-n}\!|Du|^k\eta(t)h_i\langle (D^2u)x,e_i\rangle}{e^{E+A\frac{\rho^2}{2}}}\\
&-\frac{(k\!+\!1)h^2\!\rho^{p-n}\!|Du|^k\bigg(\!\langle(D^2u)x,x\rangle \!+\!(\langle |Du|^{-1}Du D^2u,x\rangle)\!+\!|Du|\!\bigg)(e^{E+A\frac{\rho^2}{2}}\!-\!\eta(t)h)}{e^{E+A\frac{\rho^2}{2}}}\bigg]\\
&-(n-k)\eta(t)+A\frac{\eta(t)\rho^2}{2}+A\bigg(\frac{\eta(t)\rho^2}{2}-he^{E+A\frac{\rho^2}{2}}\bigg)\\
\!\geq &\frac{\sigma_{n-k}}{f(x)}\bigg[\frac{(p\!+\!2\!-\!n)h\rho^{p-n}\!|Du|^{k+1}(e^{E+A\frac{\rho^2}{2}}\!-\!\eta(t)h)}{e^{E+A\frac{\rho^2}{2}}}\!\\
&-\frac{(k\!+\!1)h^2\!\rho^{p-n}\!|Du|^k\bigg(2|D^2u|+|Du|\bigg)(e^{E+A\frac{\rho^2}{2}}\!-\!\eta(t)h)}{e^{E+A\frac{\rho^2}{2}}}\bigg]\\
&-(n-k)\eta(t)+A\frac{\eta(t)\rho^2}{2}+A\bigg(\frac{\eta(t)\rho^2}{2}-he^{E+A\frac{\rho^2}{2}}\bigg)\\
=&\frac{\sigma_{n-k}}{f(x)}\frac{e^{E+A\frac{\rho^2}{2}}\!-\!\eta(t)h}{e^{E+A\frac{\rho^2}{2}}}
\bigg[(p+2-n)h\rho^{p-n}\!|Du|^{k+1}-(k+1)h^2\rho^{p-n}\!|Du|^{k+1}\\
&-2(k\!+\!1)h^2\!\rho^{p-n}\!|Du|^k|D^2u|\bigg]+\frac{\eta(t)\rho^2}{2}\bigg(A-\frac{2(n-k)}{\rho^2}\bigg)+A\bigg(\frac{\eta(t)\rho^2}{2}-he^{E+A\frac{\rho^2}{2}}\bigg).\\
\end{align*}

We have obtained the uniform bound of $|Du(X(x,t),t)|$ in proof of Lemma \ref{lem54}, by the virtue of Schauder's theory (see example Chapter 17 in \cite{GI}), we are easy to obtain $|D^2u(X(x,t),t)|\leq \widehat{C}$ ($\widehat{C}$ is a positive constant independent of $t$) on $S^{n-1}\times [0,T)$.

Now, choose $A>\max\frac{2(n-k)}{\rho^2}$. Denote
\begin{align*}
L_1=\frac{e^{E+A\frac{\rho^2}{2}}\!-\!\eta(t)h}{e^{E+A\frac{\rho^2}{2}}},
\end{align*}
\begin{align*}
L_2=\bigg[(p+2-n)h\rho^{p-n}\!|Du|^{k+1}-(k+1)h^2\rho^{p-n}\!|Du|^{k+1}-2(k\!+\!1)h^2\!\rho^{p-n}\!|Du|^k|D^2u|\bigg],
\end{align*}
\begin{align*}
L_3=\frac{\eta(t)\rho^2}{2}\bigg(A-\frac{3(n-k)}{\rho^2}\bigg),
\end{align*}
\begin{align*}
L_4=A\bigg(\frac{\eta(t)\rho^2}{2}-he^{E+A\frac{\rho^2}{2}}\bigg).
\end{align*}
Thus when $p<n-2$, if $E$ becomes appropriately negative, namely
\begin{align*}
E<\min\bigg\{\log\frac{\eta(t)\rho^2}{2h}-(n-k),\log\eta(t)h-(n-k)\bigg\}.
\end{align*}
Hence there are $L_1<0,L_2<0,L_3>0,L_4>0$. Then
\begin{align*}
\frac{\partial E}{\partial t}\geq \frac{\sigma_{n-k}}{f(x)}L_1L_2+L_3+L_4>0,
\end{align*}
thus $E$ has the uniform lower bound. Therefore we obtain the uniform lower bound of $\sigma_{n-k}$.
\end{proof}

\begin{lemma}\label{lem56} Let $p< n-2$, under the conditions of Lemma \ref{lem52}, then there is a positive constant $C_1$ independent of $t$ such that
\begin{align*}
\sigma_{n-k}\leq C_1.
\end{align*}
\end{lemma}
\begin{proof}
Combing the curvature flow (\ref{eq107}), we considering a following auxiliary function which was first introduced by Ivaki \cite{IMN} and subsequently applied to Orlicz-Minkowski flows \cite{BP}.

\begin{align*}
M=\frac{1}{1-\beta\frac{\rho^2}{2}}\frac{G\sigma_{n-k}}{h},
\end{align*}
where $\beta$ is a positive constant such that $2\beta\leq\rho^2\leq\frac{2}{\beta}$ for all $t\in[0,T)$ (know from Lemma \ref{lem52}). Suppose $(x_1,t)$ is a  spatial maximum value point of $M$. Then at point $(x_1,t)$,
\begin{align}\label{eq512}
\nabla_iM=0, \quad \text{i.e.}\quad \frac{\beta}{1-\beta\frac{\rho^2}{2}}\nabla_i\bigg(\frac{\rho^2}{2}\bigg)\frac{G\sigma_{n-k}}{h}+\nabla_i\bigg(\frac{G\sigma_{n-k}}{h}\bigg)=0,
\end{align}
and
\begin{align}\label{eq513}
\nabla_{ij}M\leq 0.
\end{align}
Now, we estimate $M$, from (\ref{eq513}), we obtain
\begin{align}\label{eq514}
\frac{\partial M}{\partial t}\leq &\frac{\partial M}{\partial t}-Gd_{ij}\nabla_{ij}M\\
\nonumber=&\frac{\partial \bigg(\frac{1}{1-\beta\frac{\rho^2}{2}}\frac{G\sigma_{n-k}}{h}\bigg)}{\partial t}-Gd_{ij}\nabla_{ij}\bigg(\frac{1}{1-\beta\frac{\rho^2}{2}}\frac{G\sigma_{n-k}}{h}\bigg)\\
\nonumber=&\frac{1}{1-\beta\frac{\rho^2}{2}}\bigg[\frac{\partial\bigg(\frac{G\sigma_{n-k}}{h}\bigg)}{\partial t}-Gd_{ij}\nabla_{ij}\bigg(\frac{G\sigma_{n-k}}{h}\bigg)\bigg]\\
\nonumber&+\frac{\beta}{\bigg(1-\beta\frac{\rho^2}{2}\bigg)^2}\frac{G\sigma_{n-k}}{h}\bigg[\frac{\partial(\frac{\rho^2}{2})}{\partial t}-Gd_{ij}\nabla_{ij}\bigg(\frac{\rho^2}{2}\bigg)\bigg]\\
\nonumber&-2Gd_{ij}\frac{\beta}{\bigg(1-\beta\frac{\rho^2}{2}\bigg)^2}\nabla_j\bigg(\frac{G\sigma_{n-k}}{h}\bigg)\nabla_i\bigg(\frac{\rho^2}{2}\bigg)\\
\nonumber&-2Gd_{ij}\frac{\beta^2}{\bigg(1-\beta\frac{\rho^2}{2}\bigg)^3}\frac{G\sigma_{n-k}}{h}\nabla_i\bigg(\frac{\rho^2}{2}\bigg)\nabla_j\bigg(\frac{\rho^2}{2}\bigg)\\
\nonumber=&\frac{1}{1-\beta\frac{\rho^2}{2}}\bigg[\frac{\partial\bigg(\frac{G\sigma_{n-k}}{h}\bigg)}{\partial t}-Gd_{ij}\nabla_{ij}\bigg(\frac{G\sigma_{n-k}}{h}\bigg)\bigg]\\
\nonumber&+\frac{\beta}{\bigg(1-\beta\frac{\rho^2}{2}\bigg)^2}\frac{G\sigma_{n-k}}{h}\bigg[\frac{\partial(\frac{\rho^2}{2})}{\partial t}-Gd_{ij}\nabla_{ij}\bigg(\frac{\rho^2}{2}\bigg)\bigg]\\
\nonumber&-2Gd_{ij}\frac{\beta}{\bigg(1-\beta\frac{\rho^2}{2}\bigg)^2}\nabla_i\bigg(\frac{\rho^2}{2}\bigg)\bigg[\nabla_j\bigg(\frac{G\sigma_{n-k}}{h}\bigg)
+\frac{\beta}{1-\beta\frac{\rho^2}{2}}\frac{G\sigma_{n-k}}{h}\nabla_j\bigg(\frac{\rho^2}{2}\bigg)\bigg].
\end{align}
From (\ref{eq512}), we can simplify (\ref{eq514}) to
\begin{align}\label{eq515}
\frac{\partial M}{\partial t}\leq &\frac{1}{1-\beta\frac{\rho^2}{2}}\bigg[\frac{\partial\bigg(\frac{G\sigma_{n-k}}{h}\bigg)}{\partial t}-Gd_{ij}\nabla_{ij}\bigg(\frac{G\sigma_{n-k}}{h}\bigg)\bigg]\\
\nonumber&+\frac{\beta}{\bigg(1-\beta\frac{\rho^2}{2}\bigg)^2}\frac{G\sigma_{n-k}}{h}\bigg[\frac{\partial(\frac{\rho^2}{2})}{\partial t}-Gd_{ij}\nabla_{ij}\bigg(\frac{\rho^2}{2}\bigg)\bigg].
\end{align}
Now, we calculate
\begin{align*}
&\frac{\partial\bigg(\frac{G\sigma_{n-k}}{h}\bigg)}{\partial t}-Gd_{ij}\nabla_{ij}\bigg(\frac{G\sigma_{n-k}}{h}\bigg)\\
=&\frac{\sigma_{n-k}\frac{\partial G}{\partial t}+G\frac{\partial \sigma_{n-k}}{\partial t}}{h}-\frac{G\sigma_{n-k}h_t}{h^2}-Gd_{ij}\frac{\nabla_{ij}[G\sigma_{n-k}]}{h}+Gd_{ij}\frac{G\sigma_{n-k}\nabla_{ij}h}{h^2}\\
&+2Gd_{ij}\sigma_{n-k}\frac{\nabla_i[G\sigma_{n-k}]\nabla_jh}{h^2}-2Gd_{ij}\frac{[G\sigma_{n-k}]\nabla_ih\nabla_jh}{h^3}\\
=&\frac{\sigma_{n-k}\frac{\partial G}{\partial t}+Gd_{ij}[(G\sigma_{n-k}-\eta(t)h)_{ij}+h_t\delta_{ij}]}{h}-\frac{G\sigma_{n-k}h_t}{h^2}-Gd_{ij}\frac{\nabla_{ij}[G\sigma_{n-k}]}{h}\\
&+Gd_{ij}\frac{G\sigma_{n-k}\nabla_{ij}h}{h^2}-2Gd_{ij}\frac{[G\sigma_{n-k}]\nabla_ih\nabla_jh}{h^3}\\
=&\frac{\frac{\partial G}{\partial t}\sigma_{n-k}-Gd_{ij}[\eta(t)\omega_{ij}-G\sigma_{n-k}\delta_{ij}]}{h}-\frac{G\sigma_{n-k}(G\sigma_{n-k}-\eta(t)h)}{h^2}\\
&+Gd_{ij}\frac{G\sigma_{n-k}\nabla_{ij}h}{h^2}-2Gd_{ij}\frac{[G\sigma_{n-k}]\nabla_ih\nabla_jh}{h^3}\\
=&\frac{\frac{\partial G}{\partial t}\sigma_{n-k}-(n-k)\eta(t)G\sigma_{n-k}+G^2\sigma_{n-k}d_{ij}\delta_{ij}}{h}-\frac{(G\sigma_{n-k})^2}{h^2}+\frac{G\sigma_{n-k}\eta(t)}{h}\\
&+Gd_{ij}\frac{G\sigma_{n-k}\nabla_{ij}h}{h^2}-2Gd_{ij}\frac{[G\sigma_{n-k}]\nabla_ih\nabla_jh}{h^3}.
\end{align*}

From the definition of $G$ and (\ref{eq508}), we know that
\begin{align*}
&\frac{\partial G}{\partial t}=\frac{1}{f(x)}\bigg(2h\rho^{p-n}|Du|^{k+1}h_t+(p-n)h^2\rho^{p-n-1}|Du|^{k+1}\rho_t+(k+1)h^2\rho^{p-n}|Du|^k\frac{\partial|Du|}{\partial t}\bigg)\\
=&\frac{1}{f(x)}\bigg[2h\rho^{p-n}|Du|^{k+1}(G\sigma_{n-k}-\eta(t)h)+(p-n)h^2\rho^{p-n-1}|Du|^{k+1}\frac{\rho}{h}(G\sigma_{n-k}-\eta(t)h)\\
&-(k+1)h^2\rho^{p-n}|Du|^k\bigg((h_t)_i\langle (D^2u)x,e_i\rangle\!+\!\bigg(\!\langle(D^2u)x,x\rangle \!+\!(\langle |Du|^{-1}Du D^2u,x\rangle)\!+\!|Du|\!\bigg)h_t\!\bigg)\!\bigg]\\
=&\frac{1}{f(x)}\bigg[(p+2-n)h\rho^{p-n}|Du|^{k+1}(G\sigma_{n-k}-\eta(t)h)\\
&-(k+1)h^2\rho^{p-n}|Du|^k\bigg((h_t)_i\langle (D^2u)x,e_i\rangle\!+\!\bigg(\!\langle(D^2u)x,x\rangle \!+\!(\langle |Du|^{-1}Du D^2u,x\rangle)\!+\!|Du|\!\bigg)h_t\!\bigg)\!\bigg].
\end{align*}

Recall that $\rho^2=h^2+|\nabla h|^2$, then
\begin{align*}
&\frac{\partial(\frac{\rho^2}{2})}{\partial t}-Gd_{ij}\nabla_{ij}\bigg(\frac{\rho^2}{2}\bigg)\\
=&hh_t+\nabla_mh\nabla_m(h_t)-Gd_{ij}\bigg(h\nabla_{ij}h+\nabla_ih\nabla_jh+\nabla_mh\nabla_j\nabla_{mi}h+\nabla_{mi}h\nabla_{mj}h\bigg)\\
=&h(G\sigma_{n-k}-\eta(t)h)+[\sigma_{n-k}\nabla_mG\nabla_mh+G\nabla_m\sigma_{n-k}\nabla_mh-\eta(t)|\nabla h|^2]-Gd_{ij}h(\omega_{ij}-h\delta_{ij})\\
&-Gd_{ij}\nabla_ih\nabla_jh-Gd_{ij}(\omega_{mij}-h_m\delta_{ij})\nabla_mh-Gd_{ij}(\omega_{mi}-h\delta_{mi})(\omega_{mj}-h\delta_{mj})\\
=&(n+1-k)hG\sigma_{n-k}-\eta(t)\rho^2+\sigma_{n-k}\nabla_mh\nabla_mG-Gd_{ij}\omega_{mi}\omega_{mj},
\end{align*}
where we use the Codazzi equation $\omega_{imj}=\omega_{ijm}$ and the $(n-k)$-homogeneity of $\sigma_{n-k}$ in the last equality. Here
\begin{align*}
\nabla_mG=\frac{-f_m}{f^2}h|Du|^{k+1}+\frac{1}{f}h_m|Du|^{k+1}+(k+1)\frac{h}{f}|Du|^k|Du|_m,
\end{align*}
since $|Du|^2=Du\cdot Du$, then $\nabla_m|Du|^2=2\nabla_m Du\cdot Du$, thus
\begin{align*}\nabla_m|Du|=|Du|^{-1}\nabla_m Du\cdot Du\leq |D^2u|.\end{align*}
Hence we can obtain $\nabla_mG\leq \widetilde{C}$.

Substitute the above calculations into (\ref{eq515}) and we use a property $d_{ij}\omega_{im}\omega_{jm}\geq (n-k)(\sigma_{n-k})^{1+\frac{1}{n-k}}$ of $\sigma_{n-k}$ (see \cite{AB0} for details), thus
\begin{align*}
&\frac{\partial M}{\partial t}\\
\leq &\frac{1}{1-\beta\frac{\rho^2}{2}}\bigg[\frac{\frac{\partial G}{\partial t}\sigma_{n-k}-(n-k)\eta(t)G\sigma_{n-k}+G^2\sigma_{n-k}d_{ij}\delta_{ij}}{h}-\frac{(G\sigma_{n-k})^2}{h^2}+\frac{G\sigma_{n-k}\eta(t)}{h}\\
&+Gd_{ij}\frac{G\sigma_{n-k}\nabla_{ij}h}{h^2}-2Gd_{ij}\frac{[G\sigma_{n-k}]\nabla_ih\nabla_jh}{h^3}\bigg]\\
&+\frac{\beta}{\bigg(1-\beta\frac{\rho^2}{2}\bigg)^2}\frac{G\sigma_{n-k}}{h}\bigg[(n+1-k)hG\sigma_{n-k}-\eta(t)\rho^2+\sigma_{n-k}\nabla_mh\nabla_mG-Gd_{ij}\omega_{mi}\omega_{mj}\bigg]\\
=&\frac{1}{1-\beta\frac{\rho^2}{2}}\bigg\{\frac{\sigma_{n-k}}{f(x)}\bigg[(p+2-n)\rho^{p-n}|Du|^{k+1}(G\sigma_{n-k}-\eta(t)h)\\
&-(k\!+\!1)h\!\rho^{p-n}\!|Du|^k\bigg(\!(h_t)_i\langle (D^2u)x,e_i\rangle\!+\!\bigg(\langle(D^2u)x,x\rangle \!+\!(\langle |Du|^{-1}Du D^2u,x\rangle)\!+\!|Du|\bigg)h_t\!\bigg)\!\bigg]\\
&-\frac{(n-k)\eta(t)G\sigma_{n-k}}{h}+\frac{G^2\sigma_{n-k}d_{ij}\delta_{ij}}{h}-\frac{(G\sigma_{n-k})^2}{h^2}+\frac{G\sigma_{n-k}\eta(t)}{h}\\
&+Gd_{ij}\frac{G\sigma_{n-k}(\omega_{ij}-h\delta_{ij})}{h^2}-2Gd_{ij}\frac{[G\sigma_{n-k}]\nabla_ih\nabla_jh}{h^3}\bigg\}\\
&+\frac{\beta}{\bigg(1-\beta\frac{\rho^2}{2}\bigg)^2}\frac{G\sigma_{n-k}}{h}\bigg[(n+1-k)hG\sigma_{n-k}-\eta(t)\rho^2+\sigma_{n-k}\nabla_mh\nabla_mG-Gd_{ij}\omega_{mi}\omega_{mj}\bigg]\\
\leq&\frac{1}{1-\beta\frac{\rho^2}{2}}\bigg\{\frac{F}{Gf(x)}\bigg[(p+2-n)\rho^{p-n}|Du|^{k+1}(F-\eta(t)h)\\
&-(k+1)h\rho^{p-n}|Du|^k\bigg((G\sigma_{n-k}-\eta(t)h)_i\langle (D^2u)x,e_i\rangle\\
&+\bigg(\langle(D^2u)x,x\rangle +(\langle |Du|^{-1}Du D^2u,x\rangle)+|Du|\bigg)F-\eta(t)h\bigg)\bigg]\\
&-\frac{(n-k)\eta(t)F}{h}+\frac{FGd_{ij}\delta_{ij}}{h}-\frac{F^2}{h^2}+\frac{F\eta(t)}{h}\\
&+\frac{(n-k)FG}{h^2}-2Gd_{ij}\frac{F\nabla_ih\nabla_jh}{h^3}\bigg\}\\
&+\frac{\beta}{\bigg(1-\beta\frac{\rho^2}{2}\bigg)^2}\bigg[(n+1-k)F^2+\frac{F^2}{hG}\nabla_mh\nabla_mG-\frac{(n-k)FG}{h}\bigg(\frac{F}{G}\bigg)^{1+\frac{1}{n-k}}\bigg]\\
\leq&\frac{1}{1-\beta\frac{\rho^2}{2}}\bigg\{\frac{F}{Gf(x)}\bigg((p+2-n)\rho^{p-n}F|Du|^{k+1}+(2|D^2u|+|Du|)F\bigg)\\
&+\frac{FGd_{ii}}{h}+\frac{F\eta(t)}{h}+\frac{(n-k)FG}{h^2}\bigg\}\\
&+\frac{\beta}{\bigg(1-\beta\frac{\rho^2}{2}\bigg)^2}\bigg[(n+1-k)F^2+\frac{F^2}{hG}\nabla_mh\nabla_mG-\frac{n-k}{hG^{\frac{1}{n-k}}}F^{2+\frac{1}{n-k}}\bigg]\\
=&\bigg(Gd_{ii}+\eta(t)+\frac{(n-k)G}{h}\bigg)\bigg[\frac{1}{1-\beta\frac{\rho^2}{2}}\frac{F}{h}\bigg]\\
&+\bigg\{\frac{h(1-\beta\frac{\rho^2}{2})}{Gf(x)}\bigg(\!(p+2-n)\rho^{p-n}|Du|^{k+1}+2|D^2u|+|Du|\bigg)\\
&+\beta h^2\bigg((n+1-k)+\frac{\nabla_mh\nabla_mG}{hG}\!\bigg)\bigg\}\bigg[\frac{1}{1-\beta\frac{\rho^2}{2}}\frac{F}{h}\!\bigg]^2\\
&-\beta h\bigg(1-\beta\frac{\rho^2}{2}\bigg)^\frac{1}{n-k}(n-k)\bigg(\frac{h}{G}\bigg)^{\frac{1}{n-k}}\bigg[\frac{1}{1-\beta\frac{\rho^2}{2}}\frac{F}{h}\bigg]^{2+\frac{1}{n-k}}.
\end{align*}
Here
\begin{align*}
d_{ij}=\frac{\partial\sigma_{n-k}(\omega_{ij}(x))}{\partial\omega_{ij}}=\sum_{p=1}^{n}\frac{\partial\sigma_{n-k}(\omega_{ij}(x))}{\partial\lambda_p}\frac{\partial\lambda_p}{\partial \omega_{ij}}=\sum_{p=1}^{n}\sigma_{k-1}^{(p)}(\omega_{ij})v_p^iv_p^j,
\end{align*}
since $\Omega_t$ is a smooth strictly convex body with uniform bound, it's not difficult to see that $d_{ii}$ has uniform upper bound. Taking
\begin{align*}
P_1=&Gd_{ii}+\eta(t)+\frac{(n-k)G}{h}\leq C_2,\\
P_2=&\frac{h(1-\beta\frac{\rho^2}{2})}{Gf(x)}\bigg(\!(p+2-n)\rho^{p-n}|Du|^{k+1}+2|D^2u|+|Du|\bigg)\\
&+\beta h^2\bigg((n+1-k)+\frac{\nabla_mh\nabla_mG}{hG}\!\bigg)\leq C_3,\\
P_3=&\beta h\bigg(1-\beta\frac{\rho^2}{2}\bigg)^\frac{1}{n-k}(n-k)\bigg(\frac{h}{G}\bigg)^{\frac{1}{n-k}}\leq C_4.
\end{align*}
Thus at $x_1$, there exists some positive constants $C_2$, $C_3$ and $C_4$ independent of $t$ such that
\begin{align*}
\frac{\partial M}{\partial t}\leq C_1M+C_2M^2-C_3M^{2+\frac{1}{n-k}}<0
\end{align*}
provided $M$ is sufficiently large. Thus $M(x,t)$ is uniformly bounded from above, from this we can get the uniformly upper bound of $\sigma_{n-k}$ which is depends on $f$ and $n$.
\end{proof}

From \cite{UR}, we know that the eigenvalues of $\{\omega_{ij}\}$ and $\{\omega^{ij}\}$ are respectively the principal radii and principal curvatures
of $\Omega_t$, where $\{\omega^{ij}\}$ is the inverse matrix of $\{\omega_{ij}\}$. Therefore to derive a positive
upper bound of principal curvatures of $\Omega_t$ at $X(x,t)$, it is equivalent to estimate the upper bound of the eigenvalues of $\{\omega^{ij}\}$.

\begin{lemma}\label{lem57} Let $p< n-2$, under the conditions of Lemma \ref{lem52}, there exists a positive constant $C$ independent of $t$ such that
\begin{align*}
\frac{1}{C}\leq\kappa_i(\cdot,t)\leq C,\qquad i=1,\cdots,n-1.
\end{align*}
\end{lemma}
\begin{proof} For any fixed $t\in[0,T)$, we suppose that the spatial maximum of eigenvalue of matrix $\{\frac{\omega^{ij}}{h}\}$ attained at a point $x_2$ in the direction of the unit vector $e_1\in T_{x_2}S^{n-1}$. By rotation, we also choose the orthonormal vector field such that $\omega_{ij}$ is diagonal and the maximum eigenvalue of $\{\frac{\omega^{ij}}{h}\}$ is $\frac{\omega^{11}}{h}$.

Firstly, we calculate the evolution equation of $\omega_{ij}$ and $\omega^{ij}$. For convenience, we set $G=\frac{h^2\rho^{p-n}|Du|^{k+1}}{f(x)}$, then $h_t=G\sigma_{n-k}-\eta(t)h$. Since $\omega_{ij}=\nabla_{ij}h+h\delta_{ij}$, we obtain
\begin{align*}
\frac{\partial \omega_{ij}}{\partial t}=&\nabla_{ij}(h_t)+h_t\delta_{ij}\\
=&\nabla_{ij}[G\sigma_{n-k}-\eta(t)h]+\bigg(G\sigma_{n-k}-\eta(t)h\bigg)\delta_{ij}\\
=&\sigma_{n-k}\nabla_{ij}G+\nabla_iG\nabla_j\sigma_{n-k}+\nabla_i\sigma_{n-k}\nabla_jG+G\nabla_{ij}\sigma_{n-k}+G\sigma_{n-k}\delta_{ij}-\eta(t)\omega_{ij},
\end{align*}
where
\begin{align*}
\nabla_i\sigma_{n-k}=d_{mn}\nabla_i(\omega_{mn}),
\end{align*}
and
\begin{align*}
\nabla_{ij}\sigma_{n-k}=d_{mn,ls}\nabla_j(\omega_{ls})\nabla_i(\omega_{mn})+d_{mn}\nabla_{ij}(\omega_{mn}).
\end{align*}
By the Codazzi equation and the Ricci identity, we have
\begin{align*}
d_{mn}\nabla_{ij}(\omega_{mn})=&d_{mn}\nabla_{nj}(\omega_{mi})\\
=&d_{mn}\nabla_{jn}(\omega_{mi})+d_{mn}\omega_{pm}\nabla_{nj}R_{pi}+d_{mn}\delta_{pi}\nabla_{nj}R_{pm}\\
=&d_{mn}\nabla_{mn}\omega_{ij}+d_{mn}\omega_{mn}\delta_{ij}-d_{mn}\omega_{jm}\delta_{in}+d_{mn}\omega_{in}\delta_{mj}-d_{mn}\omega_{ij}\delta_{mn}\\
=&d_{mn}\nabla_{mn}\omega_{ij}+(n-k)\sigma_{n-k}\delta_{ij}-d_{mn}\delta_{mn}\omega_{ij}.
\end{align*}
Then
\begin{align*}
&\frac{\partial \omega_{ij}}{\partial t}\\
=&\sigma_{n-k}\nabla_{ij}G+\nabla_iG\nabla_j\sigma_{n-k}+\nabla_i\sigma_{n-k}\nabla_jG+(n+1-k)G\sigma_{n-k}\delta_{ij}-\eta(t)\omega_{ij}\\
&+G\bigg(d_{mn,ls}\nabla_j(\omega_{ls})\nabla_i(\omega_{mn})+d_{mn}\nabla_{mn}\omega_{ij}-d_{mn}\delta_{mn}\omega_{ij}\bigg).
\end{align*}
Hence
\begin{align}\label{eq516}
&\frac{\partial \omega_{ij}}{\partial t}-Gd_{mn}\nabla_{mn}\omega_{ij}\\
\nonumber=&\sigma_{n-k}\nabla_{ij}G+\nabla_iG\nabla_j\sigma_{n-k}+\nabla_i\sigma_{n-k}\nabla_jG+(n+1-k)G\sigma_{n-k}\delta_{ij}-\eta(t)\omega_{ij}\\
\nonumber&+G\bigg(d_{mn,ls}\nabla_j(\omega_{ls})\nabla_i(\omega_{mn})-d_{mn}\delta_{mn}\omega_{ij}\bigg).
\end{align}

Since $\frac{\partial \omega^{ij}}{\partial t}=-(\omega^{ij})^2\frac{\partial \omega_{ij}}{\partial t}$ and $\nabla_{mn}\omega^{ij}=2(\omega^{ij})^3\nabla_m\omega_{ij}\nabla_n\omega_{ij}-(\omega^{ij})^2\nabla_{mn}\omega_{ij}$, thus there is the following evolution equation by (\ref{eq516}),
\begin{align}\label{eq517}
&\frac{\partial \omega^{ij}}{\partial t}-Gd_{mn}\nabla_{mn}\omega^{ij}\\
\nonumber=&-(\omega^{ij})^2\sigma_{n-k}\nabla_{ij}G-(\omega^{ij})^2\nabla_iG\nabla_j\sigma_{n-k}
-(\omega^{ij})^2\nabla_i\sigma_{n-k}\nabla_jG\\
\nonumber&-(n+1-k)(\omega^{ij})^2G\sigma_{n-k}\delta_{ij}+\eta(t)\omega^{ij}\\
\nonumber&-G(\omega^{ij})^2\bigg(d_{mn,ls}\nabla_j(\omega_{ls})\nabla_i(\omega_{mn})-d_{mn}\delta_{mn}\omega_{ij}\bigg)-2Gd_{mn}(\omega^{ij})^3\nabla_m\omega_{ij}\nabla_n\omega_{ij}.
\end{align}

At $x_2$, we get
\begin{align}\label{eq518}
\nabla_i\frac{\omega^{11}}{h}=0,\quad \text{i.e.},\quad \omega^{11}\nabla_i\omega_{11}=-\frac{\nabla_ih}{h},
\end{align}
\begin{align*}
\nabla_{ij}\omega_{11}=&\frac{\omega_{11}\nabla_ih \nabla_jh}{h^2}-\frac{\omega_{11}\nabla_{ij}h+\nabla_ih\nabla_j\omega_{11}}{h}\\
=&2\frac{\omega_{11}\nabla_ih \nabla_jh}{h^2}-\frac{\omega_{11}\nabla_{ij}h}{h}.
\end{align*}
And
\begin{align}\label{eq519}
\nabla_{ij}\frac{\omega^{11}}{h}\leq0.
\end{align}
Now, from (\ref{eq516}) and (\ref{eq519}), we compute the following evolution equation as
\begin{align}\label{eq520}
&\frac{\partial(\frac{\omega^{11}}{h})}{\partial t}\leq\frac{\partial(\frac{\omega^{11}}{h})}{\partial t}-Gd_{ij}\nabla_{ij}\frac{\omega^{11}}{h}\\
\nonumber=&\frac{\frac{\partial \omega^{11}}{\partial t}}{h}-\frac{\omega^{11}h_t}{h^2}-Gd_{ij}\bigg(\frac{2(\omega^{11})^3\nabla_i\omega_{11}\nabla_j\omega_{11}-(\omega^{11})^2\nabla_{ij}\omega_{11}}{h}\\
\nonumber&+\frac{(\omega^{11})^2\nabla_i\omega_{11}\nabla_jh}{h^2}+\frac{(\omega^{11})^2\nabla_j\omega_{11}\nabla_ih-\omega^{11}\nabla_{ij}h}{h^2}
\nonumber+\frac{2\omega^{11}\nabla_ih\nabla_jh}{h^3}\bigg)\\
\nonumber=&\frac{-(\omega^{11})^2\nabla_{ij}G\sigma_{n-k}-(\omega^{11})^2\nabla_iG\nabla_j\sigma_{n-k}-(\omega^{11})^2\nabla_jG\nabla_i\sigma_{n-k}}{h}\\
\nonumber&-\frac{(n+1-k)(\omega^{11})^2G\sigma_{n-k}\delta_{ij}-\eta(t)\omega^{11}}{h}\\
\nonumber&-\frac{G(\omega^{11})^2\bigg(d_{ij,ls}\nabla_j(\omega_{ls})\nabla_i(\omega_{11})+d_{ij}\nabla_{ij}\omega_{11}-d_{ij}\delta_{ij}\omega_{11}\bigg)}{h}-\frac{\omega^{11}h_t}{h^2}\\
\nonumber&-Gd_{ij}\bigg(\frac{2(\omega^{11})^3\nabla_i\omega_{11}\nabla_j\omega_{11}-(\omega^{11})^2\nabla_{ij}\omega_{11}}{h}+\frac{(\omega^{11})^2\nabla_i\omega_{11}\nabla_jh}{h^2}\\
\nonumber&+\frac{(\omega^{11})^2\nabla_j\omega_{11}\nabla_ih-\omega^{11}\nabla_{ij}h}{h^2}+\frac{2\omega^{11}\nabla_ih\nabla_jh}{h^3}\bigg)\\
\nonumber=&\frac{-(\omega^{11})^2\nabla_{ij}G\sigma_{n-k}}{h}-\frac{(\omega^{11})^2(\nabla_iG\nabla_j\sigma_{n-k}+\nabla_jG\nabla_i\sigma_{n-k})}{h}\\
\nonumber&-\frac{(n+1-k)(\omega^{11})^2G\sigma_{n-k}\delta_{ij}-\eta(t)\omega^{11}}{h}\\
\nonumber&-\frac{G(\omega^{11})^2\bigg(d_{ij,ls}\nabla_j(\omega_{ls})\nabla_i(\omega_{11})+2d_{ij}\omega^{11}\nabla_i\omega_{11}\nabla_j\omega_{11}\bigg)}{h}\\
\nonumber&+G(\omega^{11})^2\frac{d_{ij}\delta_{ij}\omega_{11}}{h}-\frac{\omega^{11}h_t}{h^2}-Gd_{ij}\bigg(\frac{(\omega^{11})^2\nabla_i\omega_{11}\nabla_jh}{h^2}\\
\nonumber&+\frac{(\omega^{11})^2\nabla_j\omega_{11}\nabla_ih-\omega^{11}\nabla_{ij}h}{h^2}+\frac{2\omega^{11}\nabla_ih\nabla_jh}{h^3}\bigg).
\end{align}

By the reverse concavity of $(\sigma_{n-k})^{\frac{1}{n-k}}$ in \cite{AB0}, we have
\begin{align}\label{eq521}
(d_{ij,mn}+2d_{im}\omega^{nj})\nabla_1\omega_{ij}\nabla_1\omega_{mn}\geq \frac{n+1-k}{n-k}\frac{(\nabla_1\sigma_{n-k})^2}{\sigma_{n-k}}.
\end{align}
Moreover, according to Schwartz inequality, the following result is true,
\begin{align}\label{eq522}
2|\nabla_1\sigma_{n-k}\nabla_1G|\leq\frac{n+1-k}{n-k}\frac{G(\nabla_1\sigma_{n-k})^2}{\sigma_{n-k}}+\frac{n-k}{n+1-k}\frac{\sigma_{n-k}(\nabla_1G)^2}{G}.
\end{align}
Thus at point $x_2$, substituting (\ref{eq518}), (\ref{eq521}) and (\ref{eq522}) into (\ref{eq520}), we get
\begin{align}\label{eq523}
\nonumber&\frac{\partial(\frac{\omega^{11}}{h})}{\partial t}-G\sigma_{n-k}d_{ij}\nabla_{ij}\frac{\omega^{11}}{h}\\
\nonumber\leq&-\frac{(\omega^{11})^2\nabla_{ij}G\sigma_{n-k}}{h}+\frac{(\omega^{11})^2}{h}\bigg(\frac{n+1-k}{n-k}\frac{G(\nabla_1\sigma_{n-k})^2}{\sigma_{n-k}}
+\frac{n-k}{n+1-k}\frac{\sigma_{n-k}(\nabla_1G)^2}{G}\bigg)\\
\nonumber&-\frac{(n+1-k)(\omega^{11})^2G\sigma_{n-k}\delta_{ij}}{h}+2\frac{\eta(t)\omega^{11}}{h}\\
\nonumber&-\frac{G(\omega^{11})^2}{h}\frac{n+1-k}{n-k}\frac{(\nabla_1\sigma_{n-k})^2}{\sigma_{n-k}}+\frac{(n-k)G\sigma_{n-k}(\omega^{11})^2}{h}\\
\nonumber&-Gd_{ij}\bigg(\frac{(\omega^{11})^2\nabla_i\omega_{11}\nabla_jh}{h^2}+\frac{(\omega^{11})^2\nabla_j\omega_{11}\nabla_ih-\omega^{11}\nabla_{ij}h}{h^2}+\frac{2\omega^{11}\nabla_ih\nabla_jh}{h^3}\bigg)\\
\nonumber\leq&-\frac{(\omega^{11})^2\nabla_{ij}G\sigma_{n-k}}{h}+\frac{(\omega^{11})^2\sigma_{n-k}}{h}\bigg(\frac{n-k}{n+1-k}\frac{(\nabla_1G)^2}{G}\bigg)\\
\nonumber&-\frac{(n+1-k)(\omega^{11})^2G\sigma_{n-k}\delta_{ij}}{h}+2\frac{\eta(t)\omega^{11}}{h}\\
\nonumber&+\frac{(n-k)G\sigma_{n-k}(\omega^{11})^2}{h}+(n-k)G\sigma_{n-k}(\omega^{11})^2\bigg(\frac{\omega_{11}-h}{h^2}\bigg)\\
\leq&-\frac{(\omega^{11})^2}{h}\bigg[\nabla_{11}G\sigma_{n-k}-\frac{n-k}{n+1-k}\sigma_{n-k}\frac{(\nabla_1G)^2}{G}+(n+1-k)G\sigma_{n-k}\bigg]\\
\nonumber&+2\frac{\eta(t)\omega^{11}}{h}-(k-n-1)\frac{\omega^{11}}{h^2}G\sigma_{n-k}.
\end{align}
By using (see \cite{ZR})
\begin{align*}
\nabla_i\nabla_jG^{\frac{1}{n+1-k}}+G^{\frac{1}{n+1-k}}\delta_{ij}>0,
\end{align*}
we have
\begin{align}\label{eq524}
\frac{1}{n+1-k}\nabla_1\nabla_1G+\frac{1}{n+1-k}\bigg(\frac{1}{n+1-k}-1\bigg)\frac{(\nabla_1G)^2}{G}+G>0.
\end{align}
Inserting (\ref{eq524}) into (\ref{eq523}), by the uniform bounds on $f$, $h$, $\lambda(t)$, $|Du|$ and $\sigma_{n-k}$, we conclude there exists $c_0,c>0$ such that
\begin{align*}
\frac{\partial(\frac{\omega^{11}}{h})}{\partial t}-Gd_{ij}\nabla_{ij}\frac{\omega^{11}}{h}\leq -c_0\frac{(\omega^{11})^2}{h}+c\frac{\omega^{11}}{h}.
\end{align*}

Therefore $\omega^{11}(x,t)$ has a uniform upper bound, which means that the principal radii are bounded from below by a positive constant $c_1$. In addition, from Lemma \ref{lem56}, we know that for minimal eigenvalue $\lambda_{\min}=\frac{1}{\kappa_{\max}}$ of $\frac{\omega^{11}(x,t)}{h}$ at point $x_2$,
\begin{align*}
C_1\geq\sigma_{n-k}=&\lambda_{\max}\sigma_{n-k-1}(\lambda|\lambda_{\max})+\sigma_{n-k}(\lambda|\lambda_{\max})\\
\geq &\lambda_{\max}\sigma_{n-k-1}(\lambda|\lambda_{\max})\\
\geq &C^{n-k-1}_{n-1}\lambda_{\min}^{n-k-1}\lambda_{\max}\\
\geq &C^{n-k-1}_{n-1}c_1^{n-k-1}\lambda_{\max}
\end{align*}
for constant $C$. Consequently, the principal radii of curvature has uniform upper and lower bounds. This completes the proof of Lemma \ref{lem57}.
\end{proof}

\vskip 0pt
\section{\bf The convergence of the flow}\label{sec6}

With the help of priori estimates in the section \ref{sec5}, the long-time existence and asymptotic behaviour of flow (\ref{eq107}) are obtained, we also complete the proof of Theorem \ref{thm13}.
\begin{proof}[Proof of Theorem \ref{thm13}] Since equation (\ref{eq402}) is parabolic, we can get its short time existence. Let $T$ be the maximal time such that $h(\cdot, t)$ is a smooth strictly convex solution to equation (\ref{eq402}) for all $t\in[0,T)$. Lemmas \ref{lem52}-\ref{lem56} enable us to apply Lemma \ref{lem57} to equation (\ref{eq402}), thus we can deduce an uniformly upper bound and an uniformly lower bound for the biggest eigenvalue of $\{(h_{ij}+h\delta_{ij})(x,t)\}$. This implies
\begin{align*}
C^{-1}I\leq (h_{ij}+h\delta_{ij})(x,t)\leq CI,\quad \forall (x,t)\in S^{n-1}\times [0,T),
\end{align*}
where $C>0$ is independent of $t$. This shows that equation (\ref{eq402}) is uniformly parabolic. Estimates for the higher derivatives follow from the standard regularity theory of uniformly parabolic equations
Krylov \cite{KR}. Hence we obtain the long time existence and regularity of solutions for the flow
(\ref{eq107}). Moreover, we obtain
\begin{align}\label{eq601}
\|h\|_{C^{l,m}_{x,t}(S^{n-1}\times [0,T))}\leq C_{l,m},
\end{align}
where $C_{l,m}$ ($l, m$ are nonnegative integers pairs) are independent of $t$, then $T=\infty$. Using the parabolic comparison principle, we can attain the uniqueness of smooth non-even solutions $h(\cdot,t)$ of equation (\ref{eq402}).

From the property of non-decreasing of $\widetilde{Q}_{k,{n-p}}(\Omega_t)$ in Lemma \ref{lem41}, we know that

\begin{align}\label{eq602}
\frac{\partial\widetilde{Q}_{k,{n-p}}(\Omega_t)}{\partial t}\geq 0.
\end{align}
Based on (\ref{eq602}), there exists a $t_0$ such that
\begin{align*}
\frac{\partial\widetilde{Q}_{k,{n-p}}(\Omega_t)}{\partial t}\bigg|_{t=t_0}=0,
\end{align*}
this yields
\begin{align*}
\tau (h^2+|\nabla h|^2)^{\frac{p-n}{2}}h|Du|^{k+1}\sigma_{n-k}=f.
\end{align*}
Let $\Omega=\Omega_{t_0}$, thus the support function of $\Omega$ satisfies equation (\ref{eq106}).

In view of (\ref{eq601}), applying the Arzel$\grave{a}$-Ascoli theorem \cite{BRE} and a diagonal argument, we can extract
a subsequence of $t$, it is denoted by $\{t_j\}_{j \in \mathbb{N}}\subset (0,+\infty)$, and there exists a smooth function $\bar{h}(x)$ such that
\begin{align}\label{eq603}
\|h(x,t_j)-\bar{h}(x)\|_{C^l(S^{n-1})}\rightarrow 0,
\end{align}
uniformly for each nonnegative integer $l$ as $t_j \rightarrow \infty$. This reveals that $\bar{h}(x)$ is a support function. Let us denote by $\Omega$ the convex body determined by $\bar{h}(x)$. Thus $\Omega$ is a smooth strictly convex body containing the origin in its interior point.

Moreover, by (\ref{eq601}) and the uniform estimates in Section \ref{sec5}, we conclude that $\widetilde{Q}_{k,{n-p}}(\Omega_t)$ is a bounded function in $t$ and $\frac{\partial \widetilde{Q}_{k,{n-p}}(\Omega_t)}{\partial t}$ is uniformly continuous. Thus for any $t>0$, by the monotonicity of $\widetilde{Q}_{k,{n-p}}(\Omega_t)$ in Lemma \ref{lem41}, there is a constant $C>0$ independent of $t$, such that
\begin{align*}
\int_0^t\bigg(\frac{\partial\widetilde{Q}_{k,{n-p}}(\Omega_t)}{\partial t}\bigg)dt=\widetilde{Q}_{k,{n-p}}(\Omega_t)-\widetilde{Q}_{k,{n-p}}(\Omega_0))\leq C,
\end{align*}
this gives
\begin{align}\label{eq604}
\lim_{t\rightarrow\infty}\widetilde{Q}_{k,{n-p}}(\Omega_t)-\widetilde{Q}_{k,{n-p}}(\Omega_0)=
\int_0^\infty\frac{\partial}{\partial t}\widetilde{Q}_{k,{n-p}}(\Omega_t)dt\leq C.
\end{align}
The left hand side of (\ref{eq604}) is bounded, therefore there is a subsequence $t_j\rightarrow\infty$ such that
\begin{align*}
\frac{\partial}{\partial t}\widetilde{Q}_{k,{n-p}}(\Omega_{t_j})\rightarrow 0 \quad\text{as}\quad  t_j\rightarrow\infty.
\end{align*}
The proof of Lemma \ref{lem41} shows that
\begin{align}\label{eq605}
&\frac{\partial\widetilde{Q}_{k,{n-p}}(\Omega_t)}{\partial t}\bigg|_{t=t_j}\\
\nonumber=&\frac{(p+1-k)(k+2)}{n-p}\bigg[\int_{S^{n-1}}\rho^{2p+1-n-k}\frac{h}{f(x)}|Du|^{2(k+1)}\sigma_{n-k}dv\\
\nonumber&-\frac{\int_{S^{n-1}}\rho(v,t)^{p+1-k}|Du|^{k+1}dv}{\int_{S^{n-1}}f(x)dx}\int_{S^{n-1}}\rho(v,t)^{p+1-k}|Du|^{k+1}dv\bigg]\geq0.
\end{align}
Taking the limit $t_j\rightarrow\infty$, by the equality condition of (\ref{eq605}), it means that there has
\begin{align*}
\tau[(h^\infty)^2+|\nabla h^\infty|^2]^{\frac{p-n}{2}}h^\infty|Du(X^\infty)|^{k+1}\sigma_{n-k}(h^\infty_{ij}+h^\infty\delta_{ij})=f(x),
\end{align*}
which satisfies (\ref{eq106}). From (\ref{eq601}) and the Arzel$\grave{\textrm{a}}$-Ascoli theorem, we know taht $h^\infty$ is
the support function and the convex body determined by $h^\infty$ is denoted as $\Omega^\infty$. Here $X^{\infty}=\overline{\nabla} h^{\infty}$ and $\frac{1}{\tau}=\lim_{t_j\rightarrow\infty}\eta(t_j)$. This completes the proof of Theorem \ref{thm13}.
\end{proof}

\end{document}